\documentclass[11pt]{amsart}
\usepackage{amsmath,amssymb,amsthm,amscd,latexsym}
\usepackage[all]{xy}
\usepackage[dvips]{graphicx}
\usepackage{hyperref}
\usepackage{color}

\usepackage{euscript}

\usepackage{ifpdf}

\usepackage{amsfonts}
\usepackage{mathrsfs}

\input xypic
\xyoption{all}

\newcommand{\rar}{\rightarrow}
\newcommand{\lar}{\longrightarrow}
\newcommand{\llar}{-\kern-5pt-\kern-5pt\longrightarrow}
\newcommand{\surjects}{\twoheadrightarrow}

\newtheorem{Theorem}{Theorem}[section]
\newtheorem{Lemma}[Theorem]{Lemma}
\newtheorem{Corollary}[Theorem]{Corollary}
\newtheorem{Proposition}[Theorem]{Proposition}
\newtheorem{Remark}[Theorem]{Remark}
\newtheorem{Example}[Theorem]{Example}
\newtheorem{Conjecture}[Theorem]{Conjecture}
\newtheorem{Definition}[Theorem]{Definition}
\newtheorem{Question}[Theorem]{Question}

\def\sqr#1#2{{\vcenter{\hrule height.#2pt
        \hbox{\vrule width.#2pt height#1pt \kern#1pt
            \vrule width.#2pt}
        \hrule height.#2pt}}}
\def\phi{\varphi}
\def\demo{\noindent{\bf Proof. }}
\def\square{\mathchoice\sqr64\sqr64\sqr{4}3\sqr{3}3}
\def\qed{\hspace*{\fill} $\square$}

\def\zz{{\bf z}}

\def\vv{{\bf v}}

\def\hht{{\rm ht}\,}

\def\ker{{\rm ker}\,}

\def\rk{{\rm rank}\,}

\def\restr{{\kern-1pt\restriction\kern-1pt}}

\def\pp{{\mathbb P}}

\begin{document}
\setlength{\hoffset}{-2cm}

\title[Free divisors, blowup algebras, and analytic spread]{Free divisors, blowup algebras of Jacobian ideals, \\ and maximal analytic spread}

\author[R. Burity]{Ricardo Burity }
\address{Departamento de Matem\'atica, Universidade Federal da
Para\'iba, 58051-900 Jo\~ao Pessoa, Para\'iba, Brazil.}
\email{ricardo@mat.ufpb.br}

\author[C. B. Miranda-Neto]{Cleto B. Miranda-Neto}
\address{Departamento de Matem\'atica, Universidade Federal da
Para\'iba, 58051-900 Jo\~ao Pessoa, Para\'iba, Brazil.}
\email{cleto@mat.ufpb.br}

\author[Z. Ramos]{Zaqueu Ramos}
\address{Departamento de Matem\'atica, CCET, Universidade Federal de
Sergipe, 49100-000 S\~ao Cristov\~ao, SE, Brazil.}
\email{zaqueu@mat.ufs.br}

\thanks{{\it Corresponding author}: Cleto B. Miranda-Neto  (cleto@mat.ufpb.br).}

\begin{abstract} Free divisors form a celebrated class of hypersurfaces which has been extensively studied in the past fifteen years. Our main goal is to introduce four new families of homogeneous free divisors and investigate central aspects of the blowup algebras of their Jacobian ideals. For instance, for all families the Rees algebra and its special fiber are shown to be Cohen-Macaulay -- a desirable feature in blowup algebra theory. Moreover, we raise the problem of when the analytic spread of the Jacobian ideal of a (not necessarily free) polynomial is maximal, and we characterize this property with tools ranging from cohomology to asymptotic depth. In addition, as an application, we give an ideal-theoretic homological criterion for homaloidal divisors, i.e., hypersurfaces whose polar maps are birational. 
\end{abstract}

\maketitle

\vspace{-0.1in}

\centerline{\it Dedicated with gratitude to the memory of Professor Wolmer V. Vasconcelos,} 
\centerline{\it mentor of generations of commutative algebraists.}

\medskip

\section*{Introduction}

The well-studied theory of free divisors -- or free hypersurfaces -- has its roots in the seminal work of K. Saito \cite{S}, and in subsequent papers of H. Terao \cite{Terao, Terao2, Terao3} mostly concerned with the case of hyperplane arrangements. The original environment was the complex analytic setting, and the motivation was the computation of Gauss-Manin connections for the universal unfolding of an isolated singularity; for instance, it was proved that the discriminant in the parameter space of the universal unfolding is a free divisor. Over time, different approaches, viewpoints, and interests have emerged, including algebraic (and algebro-geometric) adaptations and even generalizations that have drawn the attention of an increasing number of researchers over the last fifteen years. The list of references is huge; see, e.g., Abe \cite{Abe}, Abe, Terao and Yoshinaga \cite{ATY}, Buchweitz and Conca \cite{B-C},
Buchweitz and Mond \cite{B-M}, Calder\'on-Moreno and Narv\'aez-Macarro \cite{Espanha}, Damon \cite{Damon}, Dimca \cite{Dimca-book, Dimca}, Dimca and Sticlaru \cite{D-S}, Miranda-Neto \cite{Cleto, Cleto2}, Schenck \cite{Surv}, Schenck, Terao and Yoshinaga \cite{STY}, Schenck and Toh\v{a}neanu \cite{S-T}, Simis and Toh\v{a}neanu \cite{Simis-T}, Toh\v{a}neanu \cite{To}, and Yoshinaga \cite{Yo}. In particular, nice references containing interesting open problems on the subject (including the celebrated Terao's Conjecture) are Dimca's book \cite{Dimca-book} and Schenck's survey \cite{Surv}.

In the present paper, the general goal is to present progress on the algebraic side of the theme, by means of various techniques. First, we explicitly describe four new families of homogeneous free divisors in standard graded polynomial rings over a field $k$ with ${\rm char}\,k=0$. Second, and in the same graded setup, we turn our angle to investigating blowup algebras of Jacobian ideals of polynomials. More precisely, we prove that for our families the Rees algebra is  Cohen-Macaulay -- i.e., from a geometric point of view, blowing-up their singular loci yields arithmetically Cohen-Macaulay schemes. Furthermore we characterize in various ways, via tools varying from (local) cohomology to asymptotic depth, the maximality of the dimension of the special fiber ring for polynomials which are no longer required to be free. The relevance of the latter lies in connections to the important theory of homaloidal divisors, i.e., homogeneous polynomials $f\in k[x_1, \ldots, x_n]$ (where, typically, the field $k$ is assumed to be algebraically closed) for which the associated polar map 
$$\mathscr{P}_f=\left(\frac{\partial f}{\partial x_1} \, : \, \cdots \, : \, \frac{\partial f}{\partial x_n}\right):\pp^{n-1}\dasharrow \pp^{n-1}$$
is birational -- i.e., $\mathscr{P}_f$ is a Cremona transformation. In fact we provide, as an application, an ideal-theoretic (also homological) homaloidness criterion. It is worth mentioning that the modern theory about such polynomials began in Ein and  Shepherd-Barron \cite{E-Barron}, where it was proved for instance that the relative invariant of a regular prehomogeneous complex vector space is homaloidal. Another classical reference is Dolgachev \cite{Dolg}.

In order to introduce the other main concepts of interest to this paper, let $k$ denote a field of characteristic zero and, given $n\geq 3$, let $R=k[x_1, \ldots, x_n]$ be a standard graded polynomial ring over $k$. Let $R_+=(x_1, \ldots, x_n)$ be the irrelevant ideal. Given a non-zero reduced homogeneous polynomial $f\in R_+^2$ (whose partial derivatives will be assumed, to avoid pathologies, to be $k$-linearly independent), it is well-known that the property of $f$ being a free divisor can be translated into saying that the corresponding Jacobian ideal $J_f=(\partial f/\partial x_1, \ldots, \partial f/\partial x_1)\subset R$ is perfect of codimension 2 -- in particular, a free $f$ must be highly singular in the sense that ${\rm codim}\,{\rm Sing}\,V(f) = 2$ regardless of $n$. Therefore, the intervention of $J_f$ in the theory is (naturally) crucial, and as a bonus this allows for an interesting link to the study of blowup algebras, particularly the traditional problem of describing ideals for which such rings are Cohen-Macaulay. Here, we are especially interested in the Rees algebra $$\mathscr{R}(J_f) \, = \, R\left[\frac{\partial f}{\partial x_1}t, \ldots, \frac{\partial f}{\partial x_n}t\right] \, \subset \, R[t]$$ and its special fiber ring $\mathscr{F}(J_f)=\mathscr{R}(J_f)\otimes_Rk$, which, as is well-known from blowup theory, encode relevant geometric information. Recall that the analytic spread of $J_f$, denoted $\ell(J_f)$, is the Krull dimension of $\mathscr{F}(J_f)$, which is bounded above by $n$. Saying that $J_f$ has {\it maximal} analytic spread means $\ell(J_f)=n$.

Next we briefly describe the contents of each section of the paper.

Section \ref{Prelims} invokes the definitions that are central to this paper, such as the notions of free divisor and blowup algebras of ideals, as well as a few auxiliary facts which will be used in some parts of the paper. Also, some conventions are established.

In Section \ref{linfreediv} we present our first family of free divisors in $R$, with $n\geq 4$. They are reducible and, in fact, {\it linear} in the sense that in addition the Jacobian ideal $J_f$ is linearly presented, i.e., the entries of the corresponding Hilbert-Burch matrix are (possibly zero) linear forms. We also determine the defining equations of $\mathscr{F}(J_f)$ and compute the analytic spread as well as the reduction number of $J_f$. Moreover, we prove that $\mathscr{R}(J_f)$ and $\mathscr{F}(J_f)$ are  Cohen-Macaulay.

Section \ref{secondfam} describes our second family of free divisors, in an even number of at least 4 variables, and again reducible and linear in the above sense. We exhibit a well-structured minimal set of generators for the module of syzygies of $J_f$. In addition, as in the previous family -- but via different methods -- we show that $\mathscr{R}(J_f)$ and $\mathscr{F}(J_f)$ are Cohen-Macaulay (the latter is in fact shown to be a generic determinantal ring) and determine the analytic spread and the reduction number of $J_f$.

In Section \ref{thirdfam}, our third family is presented as a two-parameter family of (no longer linear) free divisors $f=f_{\alpha, \beta}$ in 3 variables and of degree $\alpha \beta$, where $\alpha, \beta \geq 2$. For the one-parameter family with $\beta =2$ (also for $(\alpha, \beta) = (2, 3)$), we show that $\mathscr{R}(J_f)$ is Cohen-Macaulay and derive that $\mathscr{F}(J_f)$ is isomorphic to a polynomial ring over $k$ (so that $J_f$ has reduction number zero). Also we prove that $f$ is reducible if $k={\mathbb C}$, and in case $k\subseteq {\mathbb R}$ we verify that $f$ is reducible if $\beta$ is odd and irreducible otherwise. In addition, if $\beta \geq 3$ is odd, we show how to derive yet another two-parameter family of free divisors $g=g_{\alpha, \beta}$ (of degree $\alpha \beta - \alpha$) from $f_{\alpha, \beta}$; for the one-parameter family with $\beta =3$, we deduce that $\mathscr{R}(J_g)$ is Cohen-Macaulay.

In Section \ref{fourth} we introduce our fourth family of free divisors, in 3 variables. Such (reducible) divisors have the linearity property -- as in the two first families -- and are constructed as the determinant of the Jacobian matrix of a set of quadrics which we associate to 3 given suitable linear forms. This family is in fact partially new, because if $k={\mathbb C}$ then its members are recovered by the well-known classification of linear free divisors in at most 4 variables, whereas on the other hand our (permanent) assumption on $k$ is that ${\rm char}\,k=0$. Furthermore, we show that $J_f$ is of linear type -- i.e., the canonical epimorphism from the symmetric algebra of $J_f$ onto $\mathscr{R}(J_f)$ is an isomorphism -- and we derive that $\mathscr{R}(J_f)$ is a complete intersection ring. 

For $f\in R$ belonging to any of the four (or five) families, we use a result from Miranda-Neto \cite{Cleto} to easily determine the Castelnuovo-Mumford regularity of the graded module ${\rm Der}_k(R/(f))$ formed by the $k$-derivations of the ring $R/(f)$. Needless to say, the regularity is an important invariant which controls the complexity of a module (being related to bounds on the degrees of syzygies), whereas the derivation module is a classical object as it collects the tangent vector fields defined on the hypersurface $V(f)\subset {\mathbb P}_k^{n-1}$. 

We close the paper with Section \ref{maxell}, where we address the question as to when, for a (not necessarily free) polynomial $f\in R$, the ideal $J_f$ has maximal analytic spread -- the relevance of this task is the already mentioned connection to the theory of homaloidal divisors. We provide a number of characterizations of when such maximality holds, including cohomological conditions on a suitable auxiliary module as well as the asymptotic depth associated to both adic and integral closure filtrations of $J_f$. We also point out that the main problems we raise in this paper appear in this section. This includes a conjecture predicting that if $f$ is linear free divisor satisfying $\ell (J_f)=n$ then $J_f$ is of linear type (the case of interest is $n\geq 5$), as well as the question of whether the reduced Hessian determinant of a homaloidal polynomial must necessarily be a (linear) free divisor. For all such problems we were motivated and guided by several examples, and the computations were performed with the aid of the program {\it Macaulay} of Bayer and Stillman \cite{Mac}.

\section{Preliminaries: Free divisors and blowup algebras}\label{Prelims}

We begin by invoking some definitions and auxiliary facts. First we establish the convention that, throughout the entire paper, $k$ denotes a field of characteristic zero. A few other conventions (including notations) will be made in this section. Let $R=k[x_1, \ldots, x_n]$ be a standard graded polynomial ring in $n\geq 3$ indeterminates $x_1, \ldots, x_n$ over $k$, and let $R_+=(x_1, \ldots, x_n)$ denote the homogeneous maximal ideal of $R$. 

\subsection{Free divisors}\label{free divisors} Fix a non-zero homogeneous polynomial $f\in R_+^2$. A {\it logarithmic derivation of $f$} is an operator $\theta =\sum_{i=1}^ng_i\partial /\partial x_i$, for homogeneous polynomials $g_1, \ldots, g_n\in R$ satisfying $$\theta (f) \, = \, \sum_{i=1}^ng_i\frac{\partial f}{\partial x_i} \, \in \, (f).$$ Geometrically, $\theta$ can be interpreted as a vector field defined on ${\mathbb P}_k^{n-1}$ that is tangent along the (smooth part of the) hypersurface $V(f)$. From now on we suppose $f$ is reduced in the usual sense that $f_{\rm red}=f$, that is, $f$ is (at most) a product of distinct irreducible factors. In addition, we assume throughout -- with no further mention -- that the partial derivatives of $f$ are $k$-linearly independent so as to prevent $f$ from being a {\it cone} (recall that a polynomial $g\in R$ is a {\it cone} if, after some linear change of coordinates, $g$ depends on at most $n-1$ variables). Denote by $T_{R/k}(f)$ the $R$-module formed by the logarithmic derivations of $f$, which is also called {\it tangential idealizer} (or {\it Saito-Terao module}) of $f$, and commonly denoted ${\rm Derlog}(-V(f))$. It is easy to see that $T_{R/k}(f)$ has (generic) rank $n$ as an $R$-module.

\begin{Definition}\rm $f$ is a {\it free divisor} if the $R$-module $T_{R/k}(f)$ is free.
\end{Definition}

This concept, which originated in \cite{S}, has been shown to be of great significance to a variety of branches in mathematics. We recall yet another classical object.

\begin{Definition}\rm If $f_{x_i}:=\partial f/\partial x_i$, $i=1, \ldots, n$, then the {\it Jacobian ideal of $f$} (also called {\it gradient ideal of $f$}) is given by $$J_f \, = \, (f_{x_1}, \ldots, f_{x_n}) \, \subset \, R.$$
\end{Definition}


Note that, because $f$ is not a cone, the ideal $J_f$ is {\it minimally} generated by the $n$ partial derivatives of $f$.
Also recall that the {\it Euler derivation} $$\varepsilon_n \, := \, \sum_{i=1}^nx_i\frac{\partial}{\partial x_i}$$ is logarithmic for the homogeneous polynomial $f$ by virtue of the well-known Euler's identity $\sum_{i=1}^nx_if_{x_i}=({\rm deg}\,f) f$. Now we remark that, since $T_{R/k}(f)$ decomposes into the direct sum of the module of syzygies of $J_f$ and the cyclic module $R\varepsilon_n$ (see \cite[Lemma 2.2]{Cleto}), a free basis of $T_{R/k}(f)$ if $f$ is a free divisor consists of the derivations corresponding to the columns of a minimal syzygy matrix of $f_{x_1}, \ldots, f_{x_n}$ together with $\varepsilon_n$.

Next we recall a useful characterization which is even adopted as the definition of free divisor by some authors,  and moreover highlights the central role that commutative algebra plays in the theory.

\begin{Lemma}$($\cite[Lemma 4.1]{Cleto}$)$\label{commut} $f\in R$ is a free divisor if and only if\
$J_f$ is a codimension $2$ perfect ideal $($equivalently, the ideal $J_f$ has projective dimension 1$)$.
    
\end{Lemma} 

In other words, $f$ is a free divisor if and only if $R/J_f$ is a Cohen-Macaulay ring and $\hht J_f=2$, where, here and in the entire paper, $\hht I$ stands for the height of an ideal $I\subset R$. It follows that the classical Hilbert-Burch theorem plays a major role in the algebraic side of free divisor theory. It is also worth mentioning that this fruitful interplay holds in a more general setting (see \cite{Cleto2}).

Below we invoke a well-known and very useful criterion of freeness detected by Saito himself in case $k={\mathbb C}$, but which is known to hold over any field of characteristic zero (see \cite[Theorem 2.4]{B-C}).

\begin{Lemma}$($\cite[Theorem 1.8(ii)]{S}, {\rm also} \cite[Theorem 8.1]{Dimca-book}$)$\label{Saito-crit} $f\in R$ is a
free divisor if and only if there exist $n$ vector fields $\theta_1, \ldots, \theta_n\in T_{R/k}(f)$ such that $${\rm det}\,[\theta_j(x_i)]_{i,j=1,\ldots, n} = \,\lambda f$$ for some non-zero $\lambda \in k$. In this case, the set $\{\theta_1, \ldots, \theta_n\}$ is a free basis of $T_{R/k}(f)$.
    
\end{Lemma} 

As already pointed out, up to elementary operations in the columns of an $n\times (n-1)$ syzygy matrix $\phi$ of $J_f$, the derivations $\theta_j$'s of the free basis above correspond to the columns of $\phi$ along with the Euler vector field $\varepsilon_n$.

There is also the following important subclass introduced in \cite{B-M}.

\begin{Definition}\rm $f$ is a {\it linear free divisor} if $f$ is a free divisor and the ideal $J_f$ is linearly presented.
\end{Definition}

Stated differently, $f$ is a linear free divisor if and only if $J_f$ admits a minimal graded $R$-free resolution of the form $$0\longrightarrow R(-n)^{n-1}\longrightarrow R(-n+1)^n\longrightarrow J_f\longrightarrow 0.$$ In particular, the degree of a linear free divisor is necessarily equal to $n$ (and thus has minimal degree, since any free divisor is seen to have degree at least $n$).

Now let us provisionally consider a more general setup. Let $S$ be any Noetherian commutative ring containing $k$. A {\it $k$-derivation of $S$} is defined as an additive map $\vartheta \colon S\rightarrow S$ which vanishes on $k$ and satisfies Leibniz' rule: $\vartheta(uv)=u\vartheta (v)+v\vartheta (u)$, for all $u, v\in S$. Such objects are collected in an $S$-module, denoted ${\rm Der}_k(S)$. In particular, if again $R=k[x_1, \ldots, x_n]$, we get the $R$-module ${\rm Der}_k(R)$, which is free on the $\partial /\partial x_i$'s. Now if $f\in R_+^2$ is as above, we can also consider the derivation module ${\rm Der}_k(R/(f))$, which can be graded as follows. First, assume that ${\rm Der}_k(R)$ is given the grading inherited from the natural ${\mathbb Z}$-grading of the Weyl algebra of $R$, so that each $\partial /\partial x_i$ has degree $-1$. We endow $T_{R/k}(f)$ with the induced grading from ${\rm Der}_k(R)$, that is, a logarithmic derivation  $\sum_{i=1}^ng_i\partial /\partial x_i\in T_{R/k}(f)$  has degree $\delta$ if $g_1, \ldots, g_n\in R$ have degree $\delta + 1$. For example, $\varepsilon_n\in [T_{R/k}(f)]_0$. Finally recall that there is an identification 
(see, e.g., \cite[Lemma 2.1]{Cleto})
$${\rm Der}_k(R/(f)) \, = \, T_{R/k}(f)/f{\rm Der}_k(R).$$ Then we let ${\rm Der}_k(R/(f))$ be graded with the grading induced from $T_{R/k}(f)$ by means of this quotient.

The next auxiliary lemma is concerned with the graded module ${\rm Der}_k(R/(f))$. Let us first recall the concept of Castelnuovo-Mumford regularity of a finitely generated graded module $E$ over the graded polynomial ring $R$. Let $0\rightarrow F_p\rightarrow \ldots \rightarrow F_0\rightarrow E\rightarrow 0$ be a minimal graded $R$-free resolution of $E$, where $F_i:=\bigoplus_{j=1}^{b_i} R(-a_{i,j})$, $i=0, \ldots, p$. Note that $p$ is the projective dimension of $E$.


\begin{Definition}\rm If
$m_i:={\rm max}\{a_{i, j}\,|\, 1\leq j\leq b_i\}$, $i=0, \ldots, p$, then the {\it Castelnuovo-Mumford regularity} of $E$ is defined as ${\rm reg}\,E  = {\rm max} \{m_i
- i \, | \, 0\leq i\leq p\}$.
\end{Definition}
 
This gives in some sense a numerical measure of the complexity of the module. There are more general definitions given in terms of sheaf and local cohomologies (which in turn are related), but the one given above suffices for our purposes in this paper. We refer, e.g., to \cite{Bayer-M} and \cite[Chapter 15]{B-S}.

\begin{Lemma}$($\cite[Corollary 2.5(i)]{Cleto}$)$\label{regder} If $f\in R$ is a
free divisor of degree $d$, then ${\rm reg}\,{\rm Der}_k(R/(f))=d-2$. In particular, if $f\in R$ is a linear
free divisor then ${\rm reg}\,{\rm Der}_k(R/(f))=n-2$.
    
\end{Lemma}

It is worth mentioning that some authors have investigated the Castelnuovo-Mumford regularity of other objects that are also ``differentially related" to $f$, such as the Milnor algebra $R/J_f$ (see \cite{BDSS}) and the module $T_{R/k}(f)$ itself (see \cite[Theorem 5.5]{D-Sid} and \cite[Section 3]{S0}).

\subsection{Blowup algebras} We close the section with a brief review on blowup algebras and a few closely related notions. We fix a homogeneous proper ideal $I$ of $R$. 

\begin{Definition}\rm The {\it Rees algebra of $I$} is the graded ring
$$\mathscr{R}(I) \, = \, \bigoplus_{i\geq 0}I^it^i \, \subset \, R[t],$$ where $t$ is an indeterminate over $R$. This $R$-algebra defines the blowup along the subscheme corresponding to $I$. The {\it special fiber ring of $I$}, sometimes dubbed {\it fiber cone of $I$}, is the special fiber of $\mathscr{R}(I)$, i.e., the (standard) graded $k$-algebra $$\mathscr{F}(I) \, = \, \mathscr{R}(I)\otimes_Rk \, \cong \, \bigoplus_{i\geq 0}I^i/R_+I^i.$$ The {\it analytic spread of $I$} is $\ell (I)={\rm dim}\, \mathscr{F}(I)$. There are bounds $\hht I\leq \ell (I)\leq n$.

Alternatively, $\mathscr{R}(I)$ can be realized as the quotient of the symmetric algebra ${\rm Sym}_RI$ (a basic construct in algebra) by its $R$-torsion submodule, which is in fact an ideal. Thus there is a natural $R$-algebra epimorphism ${\rm Sym}_RI\rightarrow \mathscr{R}(I)$. If this map is an isomorphism, $I$ is said to be {\it of linear type}. Since $R$ is in particular a domain, this is tantamount to saying that 
${\rm Sym}_RI$ is a domain as well. For instance, any ideal generated by a regular sequence is of linear type.
\end{Definition}

Next we provide a useful formula for the computation of the analytic spread by means of a Jacobian matrix (in characteristic zero, as we have permanently assumed). To this end we consider an even more concrete description of the Rees algebra (hence of its special fiber), to wit, if we fix generators $I=(f_1, \ldots, f_{\nu})\subset R=k[x_1, \ldots, x_n]$, then $\mathscr{R}(I)$ is just the $R$-subalgebra generated by $f_1t, \ldots, f_{\nu}t \in R[t]$. In the particular case where the $f_i$'s are all homogeneous of the same degree -- e.g., the partial derivatives of a homogeneous polynomial -- we can write the special fiber as $\mathscr{F}(I)  \cong  k[f_1, \ldots, f_{\nu}]$ as a $k$-subalgebra of $R$.

\begin{Lemma}$($\cite[Proposition 1.1]{Si}$)$\label{rank} Write $I=(f_1, \ldots, f_{\nu})$ and suppose all the $f_i$'s are homogeneous of the same degree. Set $\Theta := \left(\frac{\partial f_i}{\partial x_j}\right)$, $1\leq i\leq \nu,\, 1\leq j\leq n$. Then $\ell(I)={\rm rank}\,\Theta$.
    
\end{Lemma}

Finally recall that a subideal $K\subset I$ is a {\it reduction} of $I$ if the induced extension of Rees algebras $\mathscr{R}(K)\subset \mathscr{R}(I)$ is integral; equivalently, there exists $r\geq 0$ such that $I^{r+1}=KI^{r}$. The minimal such $r$ is denoted ${\rm r}_K(I)$. The reduction $K$ is {\it minimal} if it is minimal with respect to inclusion. Now
the {\it reduction number} of $I$ is defined as $${\rm r}(I) \, = \, {\rm min}\{{\rm r}_K(I) \mid \mbox{$K$ is a minimal reduction of $I$}\}.$$ For instance, it is a standard fact (as $k$ is infinite) that ${\rm r}(I)=0$ if and only if $I$ can be generated by $\ell(I)$ elements, which occurs if for example $I$ is of linear type. More generally, the following basic result gives a way to compute this number in the presence of a suitable condition on the standard graded $k$-algebra $\mathscr{F}(I)$, which can be also regarded (for the purpose of reading the Castelnuovo-Mumford regularity off a minimal graded free resolution) as a cyclic graded module over a polynomial ring $k[t_1, \ldots, t_{\nu}]$ whenever $I$ can be generated by $\nu$ forms in $R$.

\begin{Lemma}$($\cite[Proposition 1.2]{GST}$)$\label{r-reg} If $\mathscr{F}(I)$ is Cohen-Macaulay, then ${\rm r}(I)={\rm reg}\,\mathscr{F}(I)$.
\end{Lemma} 

\section{First family: linear free divisors in ${\mathbb P}^{n-1}$}\label{linfreediv} 

Before presenting our first family of free divisors as well as properties of related blowup algebras, let us record a couple of basic calculations which will be used without further mention in the proof of Theorem \ref{free_div_normal_curve} below.

\begin{Remark}\label{initials}\rm Let $S=k[w,u]$ be a standard graded polynomial ring in 2 variables $w, u$, and consider the ideal $\mathfrak{n}=(w,u)$. Given an integer $r\geq 2$, the following facts are well-known and easy to see.

\medskip

(a) The ideal $\mathfrak{n}^r=(w^r,w^{r-1}u,\ldots,wu^{r-1},u^r)$ is a perfect ideal of codimension 2, having the following $(r+1)\times r$ syzygy matrix:
	\begin{equation}\label{syz-mat}\phi_r=\left[\begin{array}{cccccccc}
	-w&0&\ldots&0\\
	u&-w&\ldots&0\\
	0&u&\ldots&0\\
	\vdots&\vdots&\ddots&\vdots\\
	0&0&\ldots&-w\\
	0&0&\ldots&u
	\end{array}\right];
	\end{equation}

\medskip	
	
(b) The presentation ideal of the Rees algebra $\mathscr{R}(\mathfrak{n}^r),$ that is, the kernel of the surjective map of $S$-algebras
	$$S[y_1,\ldots,y_{r+1}]\surjects \mathscr{R}(\mathfrak{n}^r),\quad y_i\mapsto w^{r-i+1}u^{i-1},$$
	is equal to $Q=(I_1(\underline{y}\cdot\phi_r), I_2(B))$, where $\underline{y}=\left[\begin{array}{ccccc}y_1&\cdots&y_{r+1}\end{array}\right]$  and $B=\left[\begin{array}{ccccc}y_1&\cdots&y_{r}\\y_2&\cdots&y_{r+1}\end{array}\right].$

\end{Remark}

\begin{Theorem}\label{free_div_normal_curve}
	Consider the standard graded polynomial ring $R=k[x_1,\ldots,x_n]$, where $n\geq4$. Denote $x_{n-1}=w$ and $x_n=u.$ Let
	$$f=2w^{n-1}u+\sum_{i=1}^{n-2}x_{i}w^{i-1}u^{n-i}.$$
	Then $f$ is a linear free divisor.
\end{Theorem} 
\demo  First notice that 
\begin{equation}\label{partials-of-f}
f_{x_i}=w^{i-1}u^{n-i}\quad \mbox{for each}\quad 1\leq i\leq n-2, 
\end{equation}
$$f_w=2(n-1)w^{n-2}u+\sum_{i=2}^{n-2}(i-1)x_{i}w^{i-2}u^{n-i}\quad \mbox{and} \quad f_u=2w^{n-1}+\sum_{i=1}^{n-2}(n-i)x_{i}w^{i-1}u^{n-(i+1)}.$$
In particular, the subideal $(f_{x_1},\ldots,f_{x_{n-2}})$ of $J_f$ is equal to the ideal $u^2(w,u)^{n-3}.$ Thus, if $\phi_{n-3}$ is the $(n-2)\times(n-3)$ syzygy matrix of the ideal $(w,u)^{n-3}$ (see (\ref{syz-mat})), then the columns of the $n\times(n-3)$ matrix 
$$\eta=\left[\begin{array}{c}\phi_{n-3}\\\hline\boldsymbol0\end{array}\right]$$ are syzygies of the gradient ideal $J_f$.
We also have the following equalities
\begin{equation}\label{syz_delta2}
uf_w=2(n-1)w^{n-2}u^2+\sum_{i=2}^{n-2}(i-1)x_{i}w^{i-2}u^{n-(i-1)}=
2(n-1)wf_{x_{n-2}}+\sum_{i=2}^{n-2}(i-1)x_{i}f_{x_{i-1}}
\end{equation}
and
\begin{equation}\label{syz_delta_1}
(n-1)uf_u=2(n-1)w^{n-1}u+\sum_{i=1}^{n-2}(n-1)(n-i)x_{i}w^{i-1}u^{n-i}=wf_w+\sum_{i=1}^{n-1}(n(n-i-1)+1)x_{i}f_{x_i}.
\end{equation}
Now note that \eqref{syz_delta2} and \eqref{syz_delta_1} yield two new (linear) syzygies of $J_f$, to wit, {\small$$\delta_1=\left[\begin{array}{ccccccccccccccccc}
	\alpha_2x_2&\alpha_3x_3&\cdots&\alpha_{n-2}x_{n-2}&2(n-1)w&-u&0
	\end{array}\right]^t$$}
and
{\small$$\delta_2=\left[\begin{array}{ccccccccccccccccc}\beta_1x_1&\beta_2x_2&\cdots &\beta_{n-2}x_{n-2}&w&-(n-1)u\end{array}\right]^t$$}
where $\alpha_i=i-1$ if $2\leq i\leq n-2$, and $\beta_i=n(n-i-1)+1$ whenever $1\leq i\leq n-2$. 

\medskip

\noindent{\bf Claim 1.} {\it The minimal graded free resolution of $J_f$ is
	\begin{equation}\label{res_gradient_ex_scroll}
	0\to R(-n)^{n-1}\stackrel{\psi}\lar R(-n+1)^n\to J_f\to 0
	\end{equation}
	where $\psi=\left[\begin{array}{c|c|cc}\eta&\delta_1&\delta_2\end{array}\right].$} 

From the discussion above, we already know that the sequence \eqref{res_gradient_ex_scroll} is a complex. To prove that it is in fact a minimal graded free resolution of $J_f$, it suffices to verify that $\hht I_{n-1}(\psi)\geq 2.$ Note we can write $\psi$ in the form
$$\psi=\left[\begin{array}{c|c}\phi_{n-3}&\boldsymbol{\ast}\\\hline\boldsymbol0&\Phi\end{array}\right]$$
where $\Phi=\left[\begin{array}{cc}-u&w\\0&-(n-1)u\end{array}\right].$ Thus, $\det\Phi\cdot I_{n-3}(\phi_{n-3})=u^2\cdot(w,u)^{n-3}\subset I_{n-1}(\psi).$ In particular, $u^{n-1}\in I_{n-1}(\psi)$. On the other hand, if we specialize the entries of $\psi$ via the $k$-algebra endomorphism of $R$ that fixes the variables $w, u$ and maps the remaining ones to $0$, we obtain the matrix
$$\overline{\psi}=\left[\begin{array}{ccccccccccccc}
-w&0&\ldots&0&0&0\\
u&-w&\ldots&0&0&0\\
0&u&\ldots&0&0&0\\
\vdots&\vdots&\ddots&\vdots&\vdots&\vdots\\
0&0&\ldots&-w&0&0\\
0&0&\ldots&u&2(n-1)w&0\\
0&0&\ldots&0&-u&w\\
0&0&\ldots&0&0&-(n-1)u
\end{array}\right].$$
The $(n-1)$-minor of $\overline{\psi}$ obtained by omitting the last row is $cw^{n-1}$ for a certain non-zero $c\in k.$ Therefore, the $(n-1)$-minor of $\psi$ obtained by omitting the last row has the shape $cw^{n-1}+G,$  for a suitable $G\in (x_1,\ldots,x_{n-2}).$
Hence, $(u^{n-1},cw^{n-1}+G)\subset I_{n-1}(\psi).$ Hence, $\hht I_{n-1}(\psi)\geq 2$ as desired.
\qed

\medskip

A computation shows that, in the first case covered by the theorem (i.e., $n=4$), the Jacobian ideal $J_f$ is of linear type; in particular, $\ell(J_f)=4$ and ${\rm r}(J_f)=0$. The case $n\geq 5$ is treated in Proposition \ref{spread4} below. As ingredients we consider the polynomial rings
$$T:=k[y_1,\ldots,y_{n-2},s,t]\quad \mbox{and} \quad T':=k[y_1,\ldots,y_{n-2}]$$ as well as the $k$-algebras $$A:=k[f_{x_1},\ldots,f_{x_{n-2}},f_w,f_u] \quad \mbox{and} \quad A':=k[f_{x_1},\ldots,f_{x_{n-2}}].$$ By factoring out $u^2$ from the polynomials $f_{x_1},\ldots,f_{x_{n-2}}$
(see (\ref{partials-of-f})), we see that $$ A' \, \cong \, A'' \, := \, k[w^{n-3},\,w^{n-4}u,\, \ldots,\,u^{n-3}].$$ All these rings are related via the following commutative diagram of $k$-algebras:
\begin{equation}\label{commutative_diagram_normal_curve}
\xymatrix{T\ar@{->>}[r]&A\\
	T'\ar@{^{(}->}[u]\ar@{->>}[r]&A'\ar[r]^-{\cong}\ar@{^{(}->}[u]& A''}
\end{equation}

\begin{Proposition}\label{spread4} Maintain the above notations, and let $f\in R$ be as in Theorem {\rm\ref{free_div_normal_curve}}, with $n\geq 5$. The following assertions hold: 
	\begin{enumerate}
		\item[(i)] $\mathscr{F}(J_f)\cong T/I_2\left[\begin{array}{ccccc}y_1&\cdots&y_{n-3}\\y_2&\cdots&y_{n-2}\end{array}\right]$ as $k$-algebras. In particular, $\mathscr{F}(J_f)$ is Cohen-Macaulay;
		\item[(ii)] $\ell(J_f)=4$;
		\item[(iii)] ${\rm r}(J_f)=1$;
		\item[(iv)] ${\rm reg}\,{\rm Der}_k(R/(f))=n-2$ {\rm (}also for $n=4${\rm )}.  
	\end{enumerate}
\end{Proposition}
\demo (i) Since $J_f$ is a homogeneous ideal generated in the same degree, there is an isomorphism of graded $k$-algebras $\mathscr{F}(J_f)\cong A$, so that $\mathscr{F}(J_f)\cong T/Q$ where $Q:=\ker(T\surjects A).$ By the diagram \eqref{commutative_diagram_normal_curve}, we get $Q'T\subset Q,$ where $Q':=\ker(T'\surjects A')$. From Remark \ref{initials}(b) we have
$$Q'T=I_2\left[\begin{array}{ccccc}y_1&\cdots&y_{n-3}\\y_2&\cdots&y_{n-2}\end{array}\right].$$ Hence, $\hht Q\geq \hht Q'T=n-4.$ Thus, in order to prove that $ Q = I_2\left[\begin{array}{ccccc}y_1&\cdots&y_{n-3}\\y_2&\cdots&y_{n-2}\end{array}\right]$, we must show $ \hht Q \leq n-4, $ or equivalently, $\dim A\geq 4$. Now, on the other hand, Lemma \ref{rank} gives $\dim A=\rk\Theta,$ where $\Theta$ is the Hessian matrix of $f.$ Notice that the (Jacobian) matrix $\Theta$ can be written in blocks as 
$$\Theta=\left[\begin{array}{c|c}
\boldsymbol0&{\Theta}_{w, u}\\
\hline
{\Theta}_{w, u}^t&\ast
\end{array}
\right]$$
where ${\Theta}_{w, u}$ is the $(n-2)\times 2$ Jacobian matrix of $f_{x_1},\ldots,f_{x_{n-2}}$ with respect to $w, u$. Clearly, $I_2({\Theta}_{w, u})\neq 0.$ In particular, $I_4(\Theta)\neq 0$ because $I_2({\Theta}_{w, u})^2\subset I_4(\Theta).$ Hence ${\rm rank}\,\Theta\geq 4$, so that 
$\dim A={\rm rank}\,\Theta \geq 4$, as needed. The Cohen-Macaulayness of $\mathscr{F}(J_f)$ will be confirmed below, in the proof of item (iii).

\medskip

\noindent (ii) Since $\ell(J_f)=\dim \mathscr{F}(J_f)$,  the statement follows directly from the proof of  (i).

\medskip

\noindent (iii) It is well-known that $I_2\left[\begin{array}{ccccc}y_1&\cdots&y_{n-3}\\y_2&\cdots&y_{n-2}\end{array}\right]$ is a perfect ideal with linear resolution (in fact, this ideal is resolved by the Eagon-Northcott complex). In particular, the ring $\mathscr{F}(J_f)\cong T/Q$  is Cohen-Macaulay and its Castelnuovo-Mumford regularity is $1.$ Thus, by Lemma \ref{r-reg}, ${\rm r}(J_f)={\rm reg\,}\mathscr{F}(J_f)=1.$

\medskip

\noindent (iv) By Theorem \ref{free_div_normal_curve}, $f$ is a linear free divisor. Now the assertion follows from Lemma \ref{regder}.\qed

\medskip

Our next goal is to prove that, for $f$ as above, $\mathscr{R}(J_f)$ is Cohen-Macaulay. First, we need an auxiliary lemma.

\begin{Lemma}\label{Lemma_scroll_mod_quadric} Let $k[\zz,\vv]=k[z_1,\ldots,z_{m},v_1,\ldots,v_{m}]$ be a polynomial ring, with $m\geq 2$. Consider $$F=a_1v_1z_1+\cdots+a_{m}v_{m}z_{m},$$ where  $a_1,\ldots,a_n$ are non-zero elements of $k$. Let  
$$M=\left[\begin{array}{ccccccc}z_1&z_2&\ldots&z_{m-1}\\z_2&z_3&\ldots&z_m\end{array}\right].$$ Then, $ k[\zz,\vv]/(I_2(M), F)$ is a Cohen-Macaulay  domain of codimension $m-1.$	
\end{Lemma}
\demo  By \cite[Section 4]{Eisenbud}, we have that $k[\zz,\vv]/I_2(M)$ is a Cohen-Macaulay  domain of codimension $m-2$. In particular, since $F\notin I_2(M)$, the ring $k[\zz,\vv]/(I_2(M), F)$ is  Cohen-Macaulay of codimension $m-1$, and so it remains to show that it is a domain. Clearly, we can assume $m\geq 3$.

\medskip
\noindent{\bf Claim.} {\it $z_1$ is  a $(k[\zz,\vv]/(I_2(M),F))$-regular element.}
\medskip

Suppose that $z_1$ is not $(k[\zz,\vv]/(I_2(M), F))$-regular. Then, $z_1\in \mathfrak{p}$ for some associated prime $\mathfrak{p}$ of $k[\zz,\vv]/(I_2(M),F)$, and hence in particular $(z_1, I_2(M),F)\subset \mathfrak{p}$. Now let $M_1$ be the matrix obtained from $M$ by deletion of its first column. Then it is easy to see that $(z_1, z_2, I_2(M_1), F)\subset \mathfrak{p}$. If $M_2$ is the matrix obtained by deletion of the first column of $M_1$, then $(z_1, z_2, z_3, I_2(M_2), F)\subset \mathfrak{p}$. Proceeding in this way, we get $$(z_1, z_2, \ldots, z_{m-1}, F)\subset \mathfrak{p}.$$
Since $\hht (z_1,z_2,\ldots,z_{m-1},F)=m,$ it follows that $\hht \mathfrak{p}\geq m.$ But, this is a contradiction because  $\mathfrak{p}$ is a associated prime of the Cohen-Macaulay (hence unmixed) ring $k[\zz,\vv]/(I_2(M),F)$, which has codimension $m-1$. This proves the Claim.

Finally, localizing in $z_1$ we deduce the isomorphism $$\frac{k[\zz,\vv][z_1^{-1}]}{(I_2(M),F)k[\zz,\vv][z_1^{-1}]}\cong \frac{k[\zz,v_2,\ldots,v_{m}][z_1^{-1}]}{I_2(M)k[\zz,v_2,\ldots,v_{m}][z_1^{-1}]}.$$ But the ring on the right side of the isomorphism is a domain. By the claim,  it follows that $k[\zz,\vv]/(I_2(M), F)$  is a domain.
\qed

\begin{Theorem}\label{spread4-CM}
	Let $f\in R$ be as in Theorem {\rm\ref{free_div_normal_curve}}. Then, $\mathscr{R}(J_f)$ is Cohen-Macaulay.
\end{Theorem}
\demo Consider the natural epimorphism
$$C:=k[x_1,\ldots,x_{n-2},w,u,y_1,\ldots,y_{n-2},s,t]\surjects \mathscr{R}(J_f),$$ whose kernel we denote $\mathcal{J}$. From the previous considerations (and notations), it follows that
$$\mathcal{K}:=(I_1(\gamma\cdot \psi),Q)\subset \mathcal{J}$$
where $\gamma=\left[\begin{array}{cccccccc}y_1&\cdots&y_{n-2}&s&t\end{array}\right]$. Note we can rewrite the ideal $\mathcal{K}$ as 
$$\mathcal{K}= I_2\left[\begin{array}{ccccc}u&y_1&\cdots&y_{n-3}\\w&y_2&\cdots&y_{n-2}\end{array}\right]+(G,\,H),$$ $$G=\gamma\cdot\delta_1=-su+2(n-1)y_{n-2}w+\sum_{i=2}^{n-2}\alpha_iy_{i-1}x_i\quad \mbox{and}\quad H=\gamma\cdot\delta_2=-(n-1)tu+sw+\sum_{i=1}^{n-2}\beta_iy_ix_i.$$



\medskip
\noindent{\bf Claim 1.} {\it Let $\mathcal{K}_1:=I_2\left[\begin{array}{ccccc}u&y_1&\cdots&y_{n-3}\\w&y_2&\cdots&y_{n-2}\end{array}\right]+(H)\subset C$. Then, $C/\mathcal{K}_1$ is a Cohen-Macaulay domain.}

\medskip

Denote $\mathcal{K}_0:=I_2\left[\begin{array}{ccccc}u&y_1&\cdots&y_{n-3}\\w&y_2&\cdots&y_{n-2}\end{array}\right].$ It is well-known that $C/\mathcal{K}_0$ is a Cohen-Macaulay integral domain of dimension $n+3$. Moreover, since $H\notin\mathcal{K}_0$, this polynomial must be $C/\mathcal{K}_0$-regular. Hence, the ring $C/\mathcal{K}_1=C/(\mathcal{K}_0, H)$ is Cohen-Macaulay of dimension $n+2$. In particular, $C/\mathcal{K}_{1}$ satisfies Serre's condition ${\rm S}_2$ (see the definition in Subsection \ref{when-max}). 
We claim that, even more, the ring $C/\mathcal{K}_{1}$ is normal, from which its integrality will follow. In order to show that $C/\mathcal{K}_{1}$ is normal, it remains to verify that it is locally regular in codimension 1. Note $\hht \mathcal{K}_1=n-2$. By the classical Jacobian criterion, it suffices to prove that $$\hht (\mathcal{K}_1,\,I_{n-2}(\Theta)) \, \geq \, n,$$ where $\Theta$ denotes the Jacobian matrix of $\mathcal{K}_1$. Notice that the matrix $\Theta$,  after a reordering of its columns (which obviously does not affect ideals of minors), can be written in the format $$\Theta=\left[\begin{array}{c|cccccc}\Theta'&0\\\hline \ast &\Theta''\end{array}\right].$$
Precisely, $\Theta'$ is the Jacobian matrix of $\mathcal{K}_0$ with respect the variables $w,u,y_1,\ldots,y_{n-2}$ and $\Theta''$ is the (row) Jacobian matrix of $H$ with respect to the variables $x_1,\ldots, x_{n-2}, s, t$.  In particular, $$(\mathcal{K}_1,\,I_1(\Theta'')\cdot I_{n-3}(\Theta'))\subset (\mathcal{K}_1, I_{n-2}(\Theta)).$$ Now pick a minimal prime  $\mathfrak{q}$ of $(\mathcal{K}_1,I_{n-2}(\Theta)).$ In particular, $I_1(\Theta'')\cdot I_{n-3}(\Theta')\subset \mathfrak{q}$, which yields $I_1(\Theta'')\subset \mathfrak{q}$ or $I_{n-3}(\Theta')\subset \mathfrak{q}$. If $I_1(\Theta'')\subset \mathfrak{q}$ then $\hht \mathfrak{q}\geq n$ because $I_1(\Theta'')=(y_1,\ldots,y_{n-2}, w, u).$ On the other hand, if $I_{n-3}(\Theta')\subset \mathfrak{q}$ then $(\mathcal{K}_0,I_{n-3}(\Theta'))\subset \mathfrak{q}$. But it is well-known that $$\hht (\mathcal{K}_0,\,I_{n-3}(\Theta')) \, = \, n.$$ Therefore, $\hht \mathfrak{q}\geq n$ in any case, and we get $\hht (\mathcal{K}_1,I_{n-2}(\Theta))\geq n$, as desired.

\medskip

\noindent{\bf Claim 2.} {\it $C/\mathcal{K}$ is a Cohen-Macaulay domain of dimension $n+1.$}

\medskip

By Claim 1 and its proof, $C/\mathcal{K}_1$ is a Cohen-Macaulay domain of dimension $n+2$. Thus, since $G\notin \mathcal{K}_1$, the ring $C/\mathcal{K}=C/(\mathcal{K}_1, G)$ is Cohen-Macaulay of dimension $n+1$. It remains to prove that $C/\mathcal{K}$ is a domain. First we claim that $u$ is $C/\mathcal{K}$-regular. Suppose otherwise. Then $u\in \mathfrak{p}$ for some associated prime $\mathfrak{p}$ of $C/\mathcal{K}$, which gives 
$$(u,wy_1,\ldots,wy_{n-3},I_2(N), G, H)\subset \mathfrak{p},$$
where 
$$N:=\left[\begin{array}{ccccc}y_1&\cdots&y_{n-3}\\y_2&\cdots&y_{n-2}\end{array}\right].$$
In particular,
\begin{equation}\label{aprox_p}
\mathcal{Q}_1:=(u,w,I_2(N), G, H)\subset \mathfrak{p}\quad\mbox{or}\quad \mathcal{Q}_2:=(u,y_1,\ldots,y_{n-3}, G, H)\subset \mathfrak{p}.
\end{equation}
We have
$$C/(u,w,I_2(N), H)\cong (k[x_1,\ldots,x_{n-2},y_1,\ldots,y_{n-2}]/(I_2(N),\beta_1y_1x_1+\cdots+\beta_{n-2} y_{n-2}x_{n-2}))[s,t].$$
From this isomorphism and Lemma~\ref{Lemma_scroll_mod_quadric}, the ring $C/(u,w,I_2(N), H)$ is a Cohen-Macaulay domain of dimension $(n-1)+2=n+1$. Thus, since $G\notin(u,w,I_2(N), H)$, we obtain that $C/\mathcal{Q}_1=C/(u,w,I_2(N), G, H)$ is a Cohen-Macaulay ring of dimension $n$. In particular, $\hht \mathcal{Q}_1=n$.
On the other hand,
$$C/\mathcal{Q}_2\cong k[x_1,\ldots,x_{n-2},w,y_{n-2},s,t]/(y_{n-2}w,sw+\beta_{n-2}y_{n-2}x_{n-2})$$
is a Cohen-Macaulay ring of dimension $n$, which yields $\hht \mathcal{Q}_2=n$. It follows, by \eqref{aprox_p}, that $\hht\mathfrak{p}\geq n$. This is a contradiction, because  $\mathfrak{p}$ is an associated prime of $C/\mathcal{K}$, which is Cohen-Macaulay of codimension $n-1$. So, $u$ is $C/\mathcal{K}$-regular. Now, by localizing in $u$ and setting $D:=k[x_1, \ldots, x_{n-2}, w, u, y_1, s, t]$, routine calculations give 
$$(C/\mathcal{K})[u^{-1}]\cong D[u^{-1}]/(G, H)D[u^{-1}]\cong k[x_1, \ldots, x_{n-2}, w, u, y_1][u^{-1}],$$ which is a domain. Hence, $C/\mathcal{K}$ is a domain, which proves Claim 2.


To conclude the proof of the theorem, we notice that since $\mathcal{K}\subset \mathcal{J}$ are prime ideals of the same codimension, then necessarily $\mathcal{K}=\mathcal{J}.$ In particular, $\mathscr{R}(J_f)\cong C/\mathcal{J}$ is Cohen-Macaulay.
\qed

\section{Second family: linear free divisors in ${\mathbb P}^{2n-1}$}\label{secondfam}

In order to describe our second family of free divisors, consider the standard graded polynomial ring $R=k[x_1,\ldots,x_{2n-2},w, u]$ in $2n\geq 4$ indeterminates over $k$. Let
$$f=wuq, \quad q=(x_1u-x_{2}w)(x_3u-x_4w)\cdots(x_{2(n-1)-1}u-x_{2(n-1)}w).$$ For every $1\leq i\leq n-1$, denote $q_i=q/(x_{2i-1}u-x_{2i}w)\in R$. Then,
\begin{equation}\label{derivativesxy}
f_{x_{2i-1}}=u^2wq_i \quad \mbox{and} \quad f_{x_{2i}}=-w^2uq_i \quad (1\leq i\leq n-1),
\end{equation}
\begin{equation}\label{derivativeswu}
f_w=qu-wu\sum_{i=1}^{n-1}x_{2i}q_i\quad \mbox{and}\quad f_u=qw+wu\sum_{i=1}^{n-1}x_{2i-1}q_i.
\end{equation}
Using \eqref{derivativesxy} and \eqref{derivativeswu} we easily deduce the following relations:
\begin{equation}\label{equationsfiber}
\det\left(\begin{array}{cc}
f_{x_{2i-1}}&f_{x_{2j-1}}\\
f_{x_{2i}}&f_{x_{2j}}
\end{array}\right)=0 \quad (1\leq i<j\leq n-1),
\end{equation}
\begin{equation}\label{syz1}
wf_{x_{2i-1}}+uf_{x_{2i}}=0 \quad(1\leq i\leq n-1),
\end{equation}
$$wf_w+uf_u=(n+1)f,$$
\begin{equation}\label{syz3}
x_{2i-1}f_{x_{2i-1}}+x_{2i}f_{x_{2i}}=f \quad(1\leq i\leq n-1),
\end{equation}
\begin{equation}\label{syz4}
(n+1)x_{2i-1}f_{x_{2i-1}}+(n+1)x_{2i}f_{x_{2i}}-uf_u-wf_w=0 \quad (1\leq i\leq n-1).
\end{equation}
Set $\alpha=a_1a_2,$ $\beta=b_1b_2$ and $\gamma=a_1b_2+a_2b_1$. In addition to the equalities above, we have   
{\small
\begin{eqnarray}
(n+1)\sum_{i=1}^{n-1} x_{2i}f_{x_{2i}}
&=&-(n+1)w^2u \sum_{i=1}^{n-1}x_{2i}q_i\nonumber\\
&=&(n+1)[ w(f_w-qu)]\nonumber\\
&=& n wf_w-uf_u+ (uf_u+ wf_w) -(n+1)f\nonumber\\
&=&nwf_w-uf_u. \label{syz5}
\end{eqnarray}}

Now we are in a position to prove the first result of this section.

\begin{Theorem}\label{second-fam} Maintain the above notations. The following assertions hold:
	\begin{enumerate}
		\item[{\rm (i)}] $f$ is a linear free divisor;
		\item[{\rm(ii)}] The $2n\times (2n-1)$ matrix
	\begin{equation}\label{presentmatrix}
		\psi_n=\arraycolsep=4pt 
		\medmuskip = 4mu
		\left[\begin{array}{cc|c|cc|cccc}
			w&(n+1)x_1&\ldots&0&0&0\\
			u&(n+1)x_2&\ldots&0&0&(n+1) x_2\\
			\hline
			\vdots&\vdots&\ddots&\vdots&\vdots&\vdots\\
			\hline
			0&0&\cdots&w&(n+1)x_{2n-3}&0\\
			0&0&\cdots&u&(n+1)x_{2n-2}&(n+1) x_{2n-2}\\
			\hline
			0&-w&\cdots&0&-w&-n w\\
			0&-u&\cdots&0&-u& u
		\end{array}\right]
	\end{equation}
is a syzygy matrix of $J_f$. Thus a free basis of $T_{R/k}(f)$ is $\{\theta_1, \ldots, \theta_{2n-1}, \varepsilon_{2n}\}$, where the $\theta_i$'s correspond to the columns of $\psi_n$; 
	 \item[(iii)] ${\rm reg}\,{\rm Der}_k(R/(f))=2(n-1)$.
	\end{enumerate}
\end{Theorem}
\demo (i) Consider the $2n\times 2n$ matrix
$$\mathscr{M}=\arraycolsep=4pt 
\medmuskip = 4mu
\left[\begin{array}{cc|c|cc|cccc}
w&x_1&\ldots&0&0&0&x_1\\
u&x_2&\ldots&0&0&(n+1) x_2&x_2\\
\hline
\vdots&\vdots&\ddots&\vdots&\vdots&\vdots&\vdots\\
\hline
0&0&\cdots&w&x_{2n-3}&0&x_{2n-3}\\
0&0&\cdots&u&x_{2n-2}&(n+1) x_{2n-2}&x_{2n-2}\\
\hline
0&0&\cdots&0&0&-n w&w\\
0&0&\cdots&0&0& u&u
\end{array}\right].$$ Using \eqref{syz1}, \eqref{syz3}, \eqref{syz5}, and the Euler relation, it is easy to see that
$$\nabla f\cdot \mathscr{M}=[\begin{array}{cc|c|cc|cccccc}0&f&\cdots&0& f&0&2nf\end{array}],$$
so that $\nabla f\cdot\mathscr{M}\equiv \boldsymbol0\mod f$. Moreover, $(n+1)f=\det \mathscr{M}.$ Thus, by Lemma \ref{Saito-crit} (or, in this case, by the version of Saito's criterion stated in \cite[Theorem 2.4]{B-C}), we conclude that $f$ is a linear free divisor.

\medskip

(ii) For simplicity, write $\psi_n = \psi$. By (i) and Lemma \ref{commut}, we  know that $J_f$ is a codimension 2 perfect ideal, so it suffices to prove that $\nabla f\cdot\psi=\boldsymbol0$ and that $\psi$ has maximal rank. The former follows by \eqref{syz1}, \eqref{syz4} and \eqref{syz5}. Now denote by $\Delta$ the $(2n-1)$-minor of $\psi$ obtained by omitting the $2n$-th row of $\psi.$ It is easy to see that $\Delta$ modulo $w$ is given by $(n+1)^{n}x_1x_3\cdots x_{2n-3}u^{n}.$ In particular, $\Delta$ is non-zero as well. Hence, $\psi$ has maximal rank. 

\medskip

(iii) By part (i), $f$ is a linear free divisor (in $2n$ variables). Now we apply Lemma \ref{regder}.\qed

\medskip


For the next results, we consider a set of  $2n$ variables $z_1,\ldots,z_{2n-2},s,t$ over $R$ as well as the natural epimorphism 
$$S:=k[x_1, \ldots,x_{2n-2}, w, u,z_1,\ldots,z_{2n-2},s,t]\surjects \mathscr{R}(J_f)$$ whose kernel we denote $\mathcal{J}$. By the equalities \eqref{equationsfiber} we have an  inclusion
$$I_2\left[\begin{array}{cccccccccc}z_{1}&z_{3}&\ldots&z_{2n-3}\\z_2&z_4&\ldots&z_{2n-2}\end{array}\right]\subset \mathcal{J}.$$
Therefore, 
$$\mathcal{K}:=\left(I_1(\gamma \cdot\psi_n),\,I_2\left[\begin{array}{cccccccccc}z_{1}&z_{3}&\ldots&z_{2n-3}\\z_2&z_4&\ldots&z_{2n-2}\end{array}\right]\right)\subset \mathcal{J}$$
where $\gamma=\left[\begin{array}{cccccccccc}z_{1}&\ldots&z_{2n-2} & s & t\end{array}\right].$ The generators of $I_1(\gamma \cdot\psi_n)$ are of three types:
\begin{equation}\label{type1}
wz_{2i-1}+uz_{2i} \quad(1\leq i\leq n-1),
\end{equation}
$$F_i:=(n+1)(x_{2i-1}z_{2i-1}+x_{2i}z_{2i})-ws-ut \quad(1\leq i\leq n-1),$$
$$G:= (n+1)\sum_{i=1}^{n-1} x_{2i}z_{2i}-n ws+ ut.$$ We can use the generators of type \eqref{type1} as well as the ideal $I_2\left[\begin{array}{cccccccccc}z_{1}&z_3&\ldots&z_{2n-3}\\z_2&z_4&\ldots&z_{2n-2}\end{array}\right]$ to rewrite $\mathcal{K}$ as
$$\mathcal{K}=\left(\underbrace{I_2\left[\begin{array}{cccccccccc}z_{1}&z_3&\ldots&z_{2n-3}&-u\\z_2&z_{4}&\ldots&z_{2n-2}&w\end{array}\right]}_{=:\mathcal{K}_0},F_1,\ldots,F_{n-1},G\right)$$
With this, we have $$S/\mathcal{K}\cong A[x_1,\ldots,x_{2n-2},s,t]/(F_1,\ldots,F_{n-1},G)A[x_1,\ldots,x_{2n-2},s,t]$$
where $A:=k[z_1,\ldots,z_{2n-2},w,u]/\mathcal{K}_0.$ Now, consider the $2n\times n$ matrix
$$\zeta= \arraycolsep=3pt 
	\medmuskip = 4mu
	\left[\begin{array}{c|c|c|c|cccc}
	(n+1)z_1&0&\ldots&0&0\\
	(n+1)z_2&0&\ldots&0&(n+1)z_2\\
	\hline
	\vdots&\vdots&\ddots&\vdots&\vdots\\
	\hline
	0&0&\ldots&(n+1)z_{2n-3}&0\\
	0&0&\ldots&(n+1)z_{2n-2}&(n+1) z_{2n-2}\\
	\hline
	-w&-w&\ldots&-w&-nw\\
	-u&-u&\ldots&-u& u
	\end{array}\right]$$ taken as a matrix with entries in the domain $A$. We denote by $M$ the $A$-module defined as the cokernel of $\zeta$.

\begin{Proposition}\label{symM} Maintain the above notations. Then:
	\begin{enumerate}
		\item[(i)] $M$ is an $A$-module of projective dimension $1$;
		\item[(ii)] The symmetric algebra ${\rm Sym}_AM$ is a  Cohen-Macaulay domain of dimension $2n+1$.
	\end{enumerate}
\end{Proposition}
\demo (i) Consider the complex 
\begin{equation}\label{complex}0\lar A^{n}\stackrel{\zeta}\lar A^{2n}\lar M\lar 0.\end{equation} By the well-known Buchsbaum-Eisenbud acyclicity criterion, in order to show that (\ref{complex}) is exact it suffices to confirm that $\rk \zeta=n$. To this end, consider the following $n\times n$ submatrix  of $\zeta$:
$$\eta:=\left[\begin{array}{cccccccc}(n+1)z_1&\cdots&0&0\\\vdots&\ddots&\vdots&\vdots\\0&\ldots&(n+1)z_{2n-3}&0\\
-w&\ldots&-w&-n w\end{array}\right].$$
We have $\det\eta= -n(n+1)^{n-1} z_1\cdots z_{2n-3}w\neq 0\, (\bmod \mathcal{K}_0)$. Hence, $\zeta$ has rank $n$.

\medskip


(ii) In addition to the property given in (i), recall $A$ is a Cohen-Macaulay domain and $\hht \mathcal{K}_0=n-1$. Then, because of \cite[Theorem 1.1]{Hun}, it suffices to show that $$\mathrm{ht}(I_t(\zeta)+\mathcal{K}_0)\geq 2n-t+1$$ for every $1\leq t\leq n$. For this note first that, by suitably permuting the rows of $\zeta$, we obtain a matrix $N$ of the form
$$N= \arraycolsep=2pt 
	\medmuskip = 2mu
	\left[\begin{array}{cccccccccc}
		\ast z_1&\cdots&0&\ldots&0&0\\
		\vdots&\ddots&\vdots&\ldots&\vdots&\vdots\\
		0&0&\ast z_{2i-1}&\ldots&0&0\\
		\vdots&\vdots&\vdots&\ddots&\vdots&\vdots\\
		0&0&0&\ldots&\ast z_{2n-3}&0\\
		-u&-u&-u&\ldots&-u&u\\
		\hline
		\ast z_2&\cdots&0&\ldots&0&\ast z_{2}\\
		\vdots&\ddots&\vdots&\cdots&\vdots&\vdots\\
		0&\cdots&\ast z_{2i}&\ldots&0&\ast z_{2i}\\
		\vdots&\vdots&\vdots&\ddots&\vdots&\vdots\\
		0&0&0&\ldots&\ast z_{2n-2}&\ast z_{2n-2}\\
		-w&-w&-w&\ldots&-w&-nw
	\end{array}\right]$$
where all coefficients $\ast$ are  equal to $n+1$. Let us denote the top and bottom blocks of $N$ by $N_{\rm odd}$ and $N_{\rm even}$, respectively. Our goal is to prove $\mathrm{ht}(I_t(N)+\mathcal{K}_0)\geq 2n-t+1$ whenever $1 \leq t \leq n$, where as before $$\mathcal{K}_0=I_2\left[\begin{array}{cccccccccc}z_{1}&z_3&\ldots&z_{2n-3}&-u\\z_2&z_{4}&\ldots&z_{2n-2}&w\end{array}\right].$$

Let $P$ be a prime ideal containing $I_t(N)+\mathcal{K}_0$ and having the same codimension. From $N_{\rm odd}$ it is easy to see that the ideal $C_t$ generated by  the $t$-products  of the set $\{z_{1},z_3,\ldots,z_{2n-3},z_{2n-1}:=u\}$ is contained in $P.$ But, by \cite[Section 2]{To2}, the minimal primes of $C_t$ are of the form $(z_{2j_1-1},\ldots,z_{2j_{n-t+1}-1})$ for certain $1\leq j_1<\cdots<j_{n-t+1}\leq n.$ Hence, we can suppose that $(z_{2j_1-1},\ldots,z_{2j_{n-t+1}-1})\subset P$ with $1\leq j_1<\cdots<j_{n-t+1}\leq n$.  Analogously, from $N_{\rm even}$ we can write $(z_{2i_1},\ldots,z_{2i_{n-t+1}})\subset P$ for certain $1\leq i_1<\cdots<i_{n-t+1}\leq n$ (we put $z_{2n}:=w$).

Let us assume that the following condition takes place: 
$$(\dag)\hspace*{2.5cm}\mathrm{There}~\mathrm{exists}~ j\in \{1, \ldots, n\}~\mathrm{such}~\mathrm{that}~\{z_{2j-1},z_{2j}\}\cap P=\{z_{2j-1}\}~\mbox{or}~\{z_{2j}\}.$$ 

Suppose $\{z_{2j-1},z_{2j}\}\cap P=\{z_{2j-1}\}.$ Then, by the relations in $\mathcal{K}_0$, all the odd variables belong to $P$. Therefore, we have $$\underbrace{(n-t+1)}_{\mbox{even variables}}+\underbrace{n}_{\mbox{odd variables}}= \, 2n-t+1~\mathrm{variables}~\mathrm{in}~ P,$$ which gives $\mathrm{ht}(I_t(N)+\mathcal{K}_0)\geq 2n-t+1$ for $1 \leq t \leq n$. The argument for the case $\{z_{2j-1},z_{2j}\}\cap P=\{z_{2j}\}$  is similar.


Now, suppose that $(\dag)$ is not true. Without loss of generality, we may assume  $(j_1, \ldots, j_{n-t+1}) =  (1, \ldots, n-t+1)$. It follows that $z_{1},z_{2},\ldots,z_{2(n-t+1)-1},z_{2(n-t+1)}\in P.$  Consider the following $t\times t$ submatrix of $\zeta$:
$$\arraycolsep=2pt 
	\medmuskip = 2mu
	\left[\begin{array}{ccccccccc}
		0&\ast z_{2(n-t+1)+1}&0&\ldots&0&0\\
		0&0&\ast z_{2(n-t+2)+1}&\ldots&0&0\\
		\vdots&	\vdots&\vdots&\ddots&\vdots&\vdots\\
		0&0&0&\ldots&\ast z_{2n-3}&0\\
		-u&-u&-u&\ldots&-u&u\\
		-w&-w&-w&\ldots&-w& -nw
	\end{array}\right].$$
The determinant of this matrix is  $cz_{2(n-t+1)+1}\cdots z_{2n-3}z_{2n-1}z_{2n}$ for some $c\in k$; in particular, this determinant lies in $P$ and hence $z_s\in P$ for some $s$ with $2(n-t+1)+1\leq s\leq 2n$.  Therefore, since we are assuming that $(\dag)$ does not hold, there exist two consecutive indices $2j-1$, $2j$ with $2(n-t+1)+1\leq 2 j-1,\, 2j\leq 2n$ satisfying $z_{2j-1},z_{2j}\in P$. Now we can suppose, without loss of generality, that $2j-1=2(n-t+1)+1$. We have
$$(z_{1},z_{2},\ldots,z_{2(n-t+1)-1},z_{2(n-t+1)},z_{2(n-t+1)+1},z_{2(n-t+1)+2})+\mathcal{K}_0\subset  P.$$ Hence, considering the subideal $$\widetilde{\mathcal{K}}_0:=I_2\left[\begin{array}{cccccccccc}z_{2(n-t+2)+1}&z_{2(n-t+3)+1}&\ldots&z_{2n-3}&-u\\z_{2(n-t+3)}&z_{2(n-t+4)}&\ldots&z_{2n-2}&w\end{array}\right] \subset \mathcal{K}_0,$$ we observe that all the variables appearing in $\widetilde{\mathcal{K}}_0$ are different from the $2(n-t+2)$ variables that already belong to $P$; since in addition $\hht \widetilde{\mathcal{K}}_0 =t-3$, we conclude $$\mathrm{ht}(I_t(N)+\mathcal{K}_0)=\hht P \geq 2(n-t+2)+(t-3)=2n-t+1$$ whenever $1 \leq t \leq n$, as needed. 

So we have shown that ${\rm Sym}_AM$ is a  Cohen-Macaulay domain. Note that $M$ possesses a rank as an $A$-module (equal to $n$, by (\ref{complex})). Now recall that the Rees algebra of the $A$-module $M$, denoted $\mathscr{R}_A(M)$, can be defined as the quotient of ${\rm Sym}_AM$ by its $A$-torsion submodule (see \cite{SUV} for the general theory). Consequently, since in this case $A$ and ${\rm Sym}_AM$ are both domains, we can identify ${\rm Sym}_AM=\mathscr{R}_A(M)$; in particular, using \cite[Proposition 2.2]{SUV} (which gives a formula for the dimension of the Rees algebra of a module with rank) and noticing that $\dim A=2n-(n-1)=n+1$, we finally get $$\dim {\rm Sym}_AM=\dim \mathscr{R}_A(M)=\dim A + {\rm rank}_AM= (n+1)+n=2n+1.$$
\qed

\begin{Theorem}\label{K=J} Maintain the above notations. Then:
	\begin{enumerate}
		\item[(i)] $\mathcal{K}=\mathcal{J}$;
		\item[(ii)] The Rees algebra $\mathscr{R}(J_f)$ is Cohen-Macaulay;
		\item[(iii)] Let  $T=k[z_1,\ldots,z_{2n-2},s,t]$, with $n\geq 3$. Then, $$\mathscr{F}(J_f)\cong T/I_2\left[\begin{array}{cccccccccc}z_{1}&z_3&\ldots&z_{2n-3}\\z_2&z_{4}&\ldots&z_{2n-2}\end{array}\right]$$ as $k$-algebras. In particular, $\mathscr{F}(J_f)$ is Cohen-Macaulay,  $\ell(J_f)=n+2$, and ${\rm r}(J_f)=1$.
	\end{enumerate}
\end{Theorem}
\demo  We have a natural epimorphism $${\rm Sym}_AM\cong S/\mathcal{K}\surjects S/\mathcal{J}\cong \mathscr{R}(J_f).$$ By Proposition \ref{symM}, $\mathcal{K}$ is a prime ideal of $S$, and $\dim {\rm Sym}_AM=2n+1=\dim R +1 =\dim \mathscr{R}(J_f)$. So $\hht \mathcal{K}= \hht \mathcal{J}$, and then $\mathcal{K}=\mathcal{J}$. Using Proposition \ref{symM} once again, we obtain that $\mathscr{R}(J_f)$ is Cohen-Macaulay. This proves (i) and (ii).

In order to prove (iii), let $R_+$ be the homogeneous maximal ideal of $R$. We have
$$\mathscr{F}(J_f)\cong \mathscr{R}(J_f)/R_+\mathscr{R}(J_f)\cong S/(R_+S, \,\mathcal{K})\cong T/I_2\left[\begin{array}{cccccccccc}z_{1}&z_3&\ldots&z_{2n-3}\\z_2&z_{4}&\ldots&z_{2n-2}\end{array}\right],$$
which, as is well-known (being a generic determinantal ring), is Cohen-Macaulay of dimension $n+2$ and moreover has regularity 1. The latter, by Lemma \ref{r-reg}, gives ${\rm r}(J_f)=1$.
\qed




\begin{Remark}\rm (a) The ideal $J_f$ is of linear type if and only if $n=2$ (i.e., the case where $R$ is a polynomial ring in $4$ variables). Indeed, by Theorem \ref{K=J}(iii), if $n\geq 3$ then ${\rm r}(J_f)=1\neq 0$, hence $J_f$ cannot be of linear type. Conversely, 
let $n=2$, so that $R=k[x_1, x_2, w, u]$. We can check that the ideals of minors of $\psi_2$ (see (\ref{presentmatrix})) satisfy $$\hht I_s(\psi_2) \, \geq \, 5 - s \, = \, (2n-1) + 2 - s  \quad\mbox{for}\quad s = 1, 2, 3.$$ It follows by \cite[Theorem 1.1]{Hun} that $J_f$ is of linear type (in particular, ${\rm r}(J_f)=0$). Note in addition that ${\rm Sym}_RJ_f$ is a complete intersection, i.e., the polynomials $$L_1=3x_1z_1+3x_2z_2-ws-ut, \, \, L_2=wz_1+uz_2, \, \, L_3=x_2z_1+x_1z_2+us+wt$$ form an $R[z_1, z_2, s, t]$-sequence.


It is also worth mentioning that, in 3 variables, if $g\in k[x, y, z]$ defines a rank 3 central hyperplane arrangement, then it has been recently shown that $J_g$ is of linear type and moreover that the property of the symmetric algebra of $J_g$ being a complete intersection  characterizes the freeness of $g$ (see \cite[Proposition 2.14 and Corollary 2.15]{BST}).

\medskip

\noindent (b) Let $n=3$, i.e., $R=k[x_1, x_2, x_3, x_4, w, u]$. In this case, a computation shows that the entries of the product $\left[\begin{array}{cccccccc}z_1&\cdots&z_4&s&t\end{array}\right]\cdot \psi_3$ form a regular sequence, i.e., ${\rm Sym}_RJ_f$ is a complete intersection once again.

\end{Remark}



\begin{Question}\label{fiber-ci}\rm  For an arbitrary $n$, is ${\rm Sym}_RJ_f$ a complete intersection ring? 
\end{Question}


\section{Third family: non-linear free plane curves}\label{thirdfam}

In this section we furnish our third family of free divisors and some of its properties. In fact, from such a family we will derive yet another one; see Remark \ref{secondfam-g}. Similar examples (also in 3 variables) can be found, e.g., in \cite{D-S} and \cite{N}.

\begin{Theorem}\label{last-family}
	Consider the two-parameter family of polynomials $$f \, = \, f_{\alpha, \beta} \, = \, (x^{\alpha}-y^{\alpha-1}z)^{\beta}+y^{\alpha\beta} \, \in \, R \, = \, k[x, y, z],$$ for integers $\alpha, \beta \geq 2$. The following assertions hold:
	\begin{enumerate}
		\item[{\rm (i)}] $f$ is a free divisor;
		\item[{\rm(ii)}] $\mathscr{R}(J_f)$ is  Cohen-Macaulay if $(\alpha, \beta) = (2, 3)$ or if $\alpha \geq 2$ and $\beta=2$;
         \item[(iii)] $J_f$ is not of linear type if $\alpha =2$ and $\beta \geq 3$ or if $\alpha \geq 3$ and $\beta =2$. In these cases, $\mathscr{F}(J_f)$ is a polynomial ring over $k$ and then ${\rm r}(J_f)=0$;
         \item[(iv)] $f$ is reducible over $k=\mathbb{C}$. If $k\subseteq \mathbb{R}$ then $f$ is reducible if $\beta$ is odd and irreducible otherwise;
        \item[(v)] ${\rm reg}\,{\rm Der}_k(R/(f))=\alpha \beta - 2$.
	\end{enumerate}
\end{Theorem}
\demo (i) We have {\small$$f_x=\alpha\beta x^{\alpha-1}(x^{\alpha}-y^{\alpha-1}z)^{\beta-1}, \quad f_y=\alpha\beta y^{\alpha\beta -1}-(\alpha-1)\beta y^{\alpha-2}z(x^{\alpha}-y^{\alpha-1}z)^{\beta-1},\,f_z=-\beta y^{\alpha-1}(x^{\alpha}-y^{\alpha-1}z)^{\beta-1}$$}
Note that we can write $f_x,\,f_y$ and $f_z$ as 
$$f_x=\alpha x^{\alpha-1}G,\quad f_y=x^{\alpha-1}P+y^{\alpha-1}Q,\quad f_{z}=-y^{\alpha-1}G$$
for certain $G,\,P,\,Q\in R$. Thus, \begin{equation}\label{minorsmatrix}J_f=I_2\left[\begin{array}{cc}y^{\alpha-1}& -\alpha^{-1}P\\0&G\\\alpha x^{\alpha-1}&Q\end{array}\right].\end{equation} In particular, since $J_f$ has codimension two, it follows by the Hilbert-Burch theorem that $J_f$ is a perfect ideal. By Lemma \ref{commut}, $f$ is a free divisor.

\medskip

(ii) In the specific cases $(\alpha, \beta)=(2,2)$ and  $(\alpha, \beta)=(2,3)$, the Cohen-Macaulayness of $\mathscr{R}(J_f)$ can be confirmed by a routine computation. Therefore, we may suppose $\alpha\geq 3$ and $\beta=2$. Determining $G, P, Q$ in this situation, we obtain from (\ref{minorsmatrix}) that a syzygy matrix of $J_f$ is
$$\phi=\left[\begin{array}{ccc}
y^{\alpha-1}&(\alpha-1)xy^{\alpha-2}z\\
0&\alpha(x^{\alpha}-y^{\alpha-1}z)\\
\alpha x^{\alpha-1}&\alpha^2y^{\alpha}+\alpha(\alpha-1)y^{\alpha-2}z^2
\end{array}\right].$$
Let us denote by ${\mathcal Q}$ the (prime) ideal of $k[x,y,z,s,t,u]=R[s, t, u]$ defining $\mathscr{R}(J_f)$. Notice that
\begin{equation}\label{eq_sym}
I_1\left(\left[\begin{array}{ccc}s&t&u\end{array}\right]\cdot \phi\right)=(sy^{\alpha-1}+\alpha ux^{\alpha-1},y^{\alpha-2}H+\alpha x^{\alpha}t)\subset {\mathcal Q},
\end{equation}
where $H:=(\alpha-1)xzs+(\alpha^2y^2+\alpha(\alpha-1)z^2)u-\alpha yzt$. Clearly, we can rewrite $I_1([\begin{array}{ccc}s&t&u\end{array}]\cdot \phi)$ as 
$$I_1\left(\left[\begin{array}{cc}y^{\alpha-2}&x^{\alpha-2}\end{array}\right]\cdot\left[\begin{array}{cc}sy&H\\\alpha ux&\alpha x^2t\end{array}\right] \right).$$
Now it follows from Cramer's rule that 
\begin{equation}\label{sylvester}
\det \left[\begin{array}{cc}sy&H\\\alpha ux&\alpha x^2t\end{array}\right]=\alpha x^2yst-\alpha uxH=\alpha x( xyst- uH)\in {\mathcal Q}.
\end{equation}
From \eqref{eq_sym} and \eqref{sylvester} we deduce an inclusion
$$\mathcal{P}:=(sy^{\alpha-1}+\alpha ux^{\alpha-1},\, y^{\alpha-2}H+\alpha x^{\alpha}t,\, xyst- uH)\subset {\mathcal Q}.$$

\medskip

\noindent {\bf Claim 1.} {\it ${\mathcal P}$ is a perfect ideal of height $2.$}

\medskip
Clearly, $\hht {\mathcal P}\geq 2$ and 
$${\mathcal P}=
I_2\left[\begin{array}{cc}
H& -xt\\
-ys&u\\
\alpha x^{\alpha-1}& y^{\alpha-2}
\end{array}\right]$$
Now, Claim 1 follows by the Hilbert-Burch theorem.

\medskip

\noindent {\bf Claim 2.} ${\mathcal P}={\mathcal Q}$.

\medskip

Since ${\mathcal P}\subseteq {\mathcal Q}$ and $\hht {\mathcal P} =\hht {\mathcal Q}=2$, it suffices to prove that the ideal ${\mathcal P}$ is prime. Note first that $x$ is regular modulo ${\mathcal P}$. To show this, suppose otherwise. Then we would have $x\in \mathfrak{p} $ for some associated prime $\mathfrak{p}$ of $R[s,t,u]/{\mathcal P}$. In particular, $(x,sy^{\alpha-1},y^{\alpha-2}H,uH)\subset \mathfrak{p}$. Using the explicit format of $H$ given above, it is easy to see that $$\hht(x, sy^{\alpha-1}, y^{\alpha-2}H, uH) \, \geq \, 3.$$ In particular, $\hht \mathfrak{p}\geq 3.$ But this is a contradiction because, by Claim 1, $\mathcal{P}$ is perfect of height 2.

Now, by inverting the element $x$ we get
\begin{eqnarray}
{\mathcal D}\, := \, \frac{k[x,y,z,s,t,u][x^{-1}]}{{\mathcal P}k[x,y,z,s,t,u][x^{-1}]}&=&\frac{k[x,y,z,s,t,u][x^{-1}]}{(u+\alpha^{-1} x^{1-\alpha}y^{\alpha-1}s,xt+\alpha^{-1} x^{1-\alpha}y^{\alpha-2}H,xyst-uH)}\nonumber\\
&=&\frac{k[x,y,z,s,t,u][x^{-1}]}{(u+\alpha^{-1} x^{1-\alpha}y^{\alpha-1}s,xt+\alpha^{-1} x^{1-\alpha}y^{\alpha-2}H)}\nonumber\\
&\cong&\frac{k[x,y,z,s,t][x^{-1}]}{(at+bs)}\nonumber\\
&\cong &\frac{k[x,y,z,x^{-1}][s,t]}{(at+bs)}\nonumber
\end{eqnarray}
where $a:=x(1- x^{-\alpha}y^{\alpha-1}z)\quad \mbox{and}\quad b:=x^{1-\alpha}y^{\alpha - 2}[(\alpha-1)za+x^{1-\alpha}y^{\alpha+1}]$ are elements in the coefficient ring $k[x,y,z,x^{-1}].$ Since $a$ and $b$ are easily seen to be relatively prime in this factorial domain, the element $at+bs$ must be irreducible in $k[x,y,z,x^{-1}][s,t]$, so that the quotient  $k[x,y,z,x^{-1}][s,t]/(at+bs)\cong {\mathcal D}$ is a domain. This means (as $x$ is regular modulo ${\mathcal P}$) that  $R[s,t,u]/{\mathcal P}$ is a domain, as needed.

\smallskip

Finally, by Claim 1 and Claim 2, we conclude that $\mathscr{R}(J_f)$ is Cohen-Macaulay. 

\medskip

(iii) First, if $\alpha =2$, a simple inspection shows that the linear type property of $J_f$ fails in case $\beta \geq 3$ (and holds if $\beta =2$), by analyzing the saturation of the ideal $\mathscr{S}$ of 2 linear forms defining ${\rm Sym}_RJ_f$ in $R[s, t, u]$ by the ideal $J_f$. The resulting ideal -- which thus defines $\mathscr{R}(J_f)$ (see, e.g., \cite[Lemma 2.11]{Cleto}) -- turns out to strictly contain $\mathscr{S}$. In addition, it is contained in $(x, y, z)R[s, t, u]$, so that $\mathscr{F}(J_f)  \cong k[s, t, u]$. As to the case where $\alpha \geq 3$ and $\beta =2$, we can use a previous calculation. Precisely, by the structure of the defining ideal $\mathcal{Q}=\mathcal{P}\subset k[x, y, z, s, t, u]$ of $\mathscr{R}(J_f)$ as obtained in item (ii), we readily get that $J_f$ is not of linear type. Moreover, by looking at the non-linear Rees equation $$xyst-uH \, \in \, (x, y, z)R[s, t, u]$$ we conclude that, once again, $\mathscr{F}(J_f)  \cong  k[s, t, u]$. In either case, $\ell(J_f)=3$ and (by Lemma \ref{r-reg}) ${\rm r}(J_f)=0$. 

\medskip

(iv) Let $k=\mathbb{C}$ and assume first that $\beta = 2m$,\, $m\geq 1$. We have
$f  = (x^{\alpha}-y^{\alpha -1}z)^{2m}+y^{2m\alpha}$
and hence, for $i=\sqrt{-1}$ and $A:=x^{\alpha}-y^{\alpha -1}z$,
\begin{equation}\label{factors}
f \, = \, (iA^m + y^{m\alpha})(-iA^m + y^{m\alpha}).
\end{equation}
Now, assume $\beta \geq 3$ is odd. Write
$f= (x^{\alpha}-y^{\alpha -1}z+y^{\alpha}-y^{\alpha})^{\beta}+y^{\alpha \beta}$ and set $B :=  x^{\alpha}-y^{\alpha -1}z+y^{\alpha}$. Thus we can rewrite
$$f \, = \, (B-y^{\alpha})^{\beta}+y^{\alpha \beta} \, = \,
[Bg+(-1)^{\beta}y^{\alpha \beta}]+y^{\alpha \beta}\, = \, Bg$$ for a suitable $g:=g_{\alpha, \beta}\in R$ of degree $\alpha(\beta -1)$. Of course, this also shows that if $k\subseteq \mathbb{R}$ and $\beta$ is odd, then $f$ is reducible over $k$. 

Finally, if $k\subseteq \mathbb{R}$ then the irreducibility of $f$ over $k$ for even $\beta$ follows from the structure of the factors described in (\ref{factors}) over the unique factorization domain ${\mathbb C}[x, y, z]$.

\medskip

(v) According to item (i), $f$ is a free divisor. Noticing that its degree is $\alpha \beta$, the assertion follows by Lemma \ref{regder}.\qed

\begin{Remark}\rm Computations strongly suggest that the cases described in item (ii) are precisely the ones where the Cohen-Macaulayness of $\mathscr{R}(J_f)$ takes place. Concerning the linear type property of $J_f$, computations also indicate that $J_f$ is not of linear type if $\alpha, \beta \geq 3$ -- cases not covered by part (iii). The (partially computer-assisted) conclusion is that $J_f$ is of linear type if and only if $\alpha = \beta =2$. 

\end{Remark}

\begin{Remark}\label{secondfam-g} \rm Here we want to point out that the form $g=g_{\alpha, \beta}\in R_{\alpha(\beta-1)}$ defined in the proof of item (iv) is also a free divisor provided that $\beta \geq 3$ is odd. First, note that $g$ is defined by means of $Bg=(B-y^{\alpha})^{\beta}+y^{\alpha\beta},$ where $B=x^{\alpha}-y^{\alpha-1}z+y^{\alpha}$. Explicitly, from $$Bg=\displaystyle(B-y^{\alpha})^{\beta}+y^{\alpha\beta}= \sum_{j=0}^{\beta}\binom{\beta}{j}B^{\beta-j}(-y^{\alpha})^j+y^{\alpha\beta}=B\cdot \left(\sum_{j=0}^{\beta-1}(-1)^j\binom{\beta}{j}B^{\beta-j-1}y^{\alpha j}\right)$$ we get $\displaystyle g=\sum_{j=0}^{\beta-1}(-1)^j\binom{\beta}{j}B^{\beta-j-1}y^{\alpha j}$. An elementary calculation shows that we can write $g_x,\,g_y$, $g_z$ as $g_x=\alpha x^{\alpha-1}T$, $g_y=x^{\alpha-1}U+y^{\alpha-1}V$, $g_{z}=y^{\alpha-1}T$, for certain $T,\,U,\,V\in R$. Thus, $$J_g=I_2\left[\begin{array}{cc}y^{\alpha-1}& -\alpha^{-1}U\\0&T\\\alpha x^{\alpha-1}&V\end{array}\right].$$
Since $\hht J_g=2$, the ideal $J_g$ must be perfect by the Hilbert-Burch theorem. Therefore, $g$ is a free divisor whenever $\beta \geq 3$ is odd. In particular, by Lemma \ref{regder} we have ${\rm reg}\,{\rm Der}_k(R/(g))=\alpha \beta - \alpha -2$. 

We also observe that, for every odd $\beta \geq 3$, the form $g$ is reducible over $k={\mathbb C}$. Indeed, let
$$\Phi =\sum_{j=0}^{\beta-1}(-1)^j\binom{\beta}{j}Z^{\beta-j-1}W^{j}\in {\mathbb C}[Z, W],$$ which then factors as a product of linear forms in $Z$ and $W$. Now let $\sigma \colon {\mathbb C}[Z, W]\rightarrow {\mathbb C}[x, y, z]$  be the homomorphism given by $Z\mapsto B$ and $W\mapsto y^{\alpha}$. Then, the form $\sigma (\Phi)=g$ is reducible in ${\mathbb C}[x, y, z]$.
	 
Finally, let us study the algebra $\mathscr{R}(J_{g_{\alpha, \beta}})$ in the cases  $\beta =3$ and $\beta=5$. 

If $\beta =3$ then first as a matter of illustration we explicitly have \begin{eqnarray}
	 g_{\alpha, 3}=B^{2}-3By^{\alpha}+3y^{2\alpha }&=&(B-y^{\alpha})^2-y^{\alpha}(B-2y^{\alpha})=(B-y^{\alpha})^2-y^{\alpha}(B-y^{\alpha})+y^{2\alpha}\nonumber\\
	 &=&(x^{\alpha}-y^{\alpha-1}z+y^{\alpha})(x^{\alpha}-y^{\alpha-1}z-2y^{\alpha})+3y^{2\alpha }\nonumber\\
	 &=&x^{2\alpha}-2x^{\alpha}y^{\alpha-1}z-x^{\alpha}y^{\alpha}+y^{2\alpha-2}z^2+y^{2\alpha-1}z+y^{2\alpha}.\nonumber
	 \end{eqnarray}
	 Computations show that, for all $\alpha \geq 2$, the ring $\mathscr{R}(J_{g_{\alpha, 3}})$ is Cohen-Macaulay, ${\rm r}(J_{g_{\alpha, 3}})=0$, but $J_f$ is of linear type if and only if $\alpha =2$; more precisely, if $\mathcal{L}$ (resp. $\mathcal{Q}$) defines ${\rm Sym}_RJ_{g_{\alpha, 3}}$ (resp. $\mathscr{R}(J_{g_{\alpha, 3}})$) in the polynomial ring $R[s, t, u]$, then we have found the relation 
	 $$\mathcal{L}: \mathcal{Q}= (x^{\alpha -1}, y^{\alpha -2})R[s, t, u].$$ 
	 
	 If $\beta =5$, then in the cases $\alpha =2$ and $\alpha =3$ we have confirmed that ${\rm depth}\,\mathscr{R}(J_{g_{\alpha, 5}})=3$, i.e., the ring $\mathscr{R}(J_{g_{\alpha, 5}})$ is {\it almost} Cohen-Macaulay in the sense that its depth is 1 less than its dimension. Also, we have ${\rm r}(J_{g_{\alpha, 5}})=0$. We strongly believe such properties hold for $\alpha \geq 4$ as well. 
	 
\end{Remark}

\begin{Question}\label{irred-g}\rm Let $\beta \geq 7$ be odd. Is $\mathscr{R}(J_{g_{\alpha, \beta}})$ almost Cohen-Macaulay? Is it true that ${\rm r}(J_{g_{\alpha, \beta}})=0$\,? If $k\subseteq \mathbb{R}$ and $\beta \geq 3$ is odd, is $g_{\alpha, \beta}$ irreducible? 
\end{Question}

\section{Fourth family: linear free plane curves}\label{fourth}

In this section, we let $R=k[x,y,z]$ and our objective is to exhibit our fourth family of free divisors, which as we shall prove have the linearity property as in two of the previous families. Although in \cite[6.4, p.\,837]{Grangeretal} a classification of linear free divisors in at most 4 variables in the case $k={\mathbb C}$ is given, the approach provided here describes concretely a recipe to detect some linear free divisors in 3 variables starting from a suitable $2\times 3$ matrix $\mathscr{L}$ of linear forms (in fact, from only 3 linear forms, as we shall clarify), where, we recall, $k$ is {\it not} required to be algebraically closed; in regard to this point, it should be mentioned that, even though most of the existing results in the literature are established over ${\mathbb C}$, there has always been an interest in free divisor theory over arbitrary fields (see, e.g., \cite{Terao4} and \cite{TP}).   

We could start focusing on the case where the 6 entries of $\mathscr{L}$ are {\it general} linear forms, but there is in fact no need for this setting as we only suppose the forms in the first row to be linearly independent over $k$ and, naturally, the rank of the matrix to be $2$. Thus, after eventually a linear change of variables, we let
$$\mathscr{L} \, = \, \mathscr{L}(L_1, L_2, L_3) \, := \, \left[\begin{array}{ccc}
	x&y&z\\
    L_1&L_2&L_3
\end{array}\right],$$ for linear forms $$L_1 \, = \, a_1x+a_2y+a_3z, \, \, \, L_2 \, = \, a_4x+a_5y+a_6z, \, \, \, L_3 \, = \, a_7x+a_8y+a_9z,$$ where at least one of the $2\times 2$ minors $$Q_1 \, = \, xL_2-yL_1, \, \, \, Q_2 \, = \, xL_3-zL_1, \, \, \, Q_3 \, = \, yL_3-zL_2$$ does not vanish. We also consider the Jacobian matrix of ${\bf Q}=\{Q_1, Q_2, Q_3\}$,
$$\Theta = \Theta({\bf Q})=\left[\begin{array}{ccc}
	{Q_1}_x&{Q_2}_x&{Q_3}_x\\
	{Q_1}_y&{Q_2}_y&{Q_3}_y\\
	{Q_1}_z&{Q_2}_z&{Q_3}_z
\end{array}\right]=\left[\begin{array}{ccc}
L_2-a_4x-a_1y&L_3+a_7x-a_1z&a_7y-a_4z\\
a_5x-L_1-a_2y&a_8x-a_2z&L_3-a_8y-a_5z\\
a_6x-a_3y&a_9x-L_1-a_3z&a_9y-L_2-a_6z
\end{array}\right].$$

Our result in this section is as follows.

\begin{Theorem}\label{fourth-f}
Maintain the above notations. If the cubic $f:=\mathrm{det}\,\Theta$ is non-zero and reduced, then $f$ is a linear free divisor. Precisely, a free basis of $T_{R/k}(f)$ is $\{\theta_1, \theta_2, \varepsilon_3\}$, where $\varepsilon_3$ is the Euler derivation, $\theta_2=L_1\frac{\partial}{\partial x}+L_2\frac{\partial}{\partial y}+L_3\frac{\partial}{\partial z}$, and $$\theta_1= (a_1L_1+a_2L_2+a_3L_3)\frac{\partial}{\partial x}+(a_4L_1+a_5L_2+a_6L_3)\frac{\partial}{\partial y}+(a_7L_1+a_8L_2+a_9L_3)\frac{\partial}{\partial z}.$$
\end{Theorem}
\begin{proof} Routine calculations show that $\eta_1$ and $\eta_2$ below are syzygies of $J_f$,
$$\eta_1=\left[\begin{array}{c}
	(2a_1-a_5-a_9)x+3a_2y+3a_3z\\
	 3a_4x+(2a_5-a_1-a_9)y+3a_6z\\
 3a_7x+3a_8y+(2a_9-a_1-a_5)z
\end{array}\right]=\left[\begin{array}{c}
3L_1 - (a_1+a_5+a_9)x\\
3L_2-(a_1+a_5+a_9)y\\
3L_3-(a_1+a_5+a_9)z
\end{array}\right]_{3\times 1}$$
and $\eta_2$ being the $3\times 1$ column-matrix given by
$$\left[\begin{array}{c}
		(-3a_5-3a_9)L_1+6a_2L_2+6a_3L_3+(-4a_6a_8-({a_5}-{a_9})^2-4a_3a_7-4a_2a_4+3a_1(a_5+a_9))x\\
		6a_4L_1+(-6a_1+3a_5-3a_9)L_2+6a_6L_3+(-4a_6a_8-({a_5}-{a_9})^2-4a_3a_7-4a_2a_4+3a_1(a_5+a_9))y\\
		6a_7L_1+6a_8L_2+(-6a_1-3a_5+3a_9)L_3+(-4a_6a_8-({a_5}-{a_9})^2-4a_3a_7-4a_2a_4+3a_1(a_5+a_9))z	
	\end{array}\right].$$
Now let $$A:=a_1+a_5+a_9, \quad B:=-4a_6a_8-({a_5}-{a_9})^2-4a_3a_7-4a_2a_4+3a_1(a_5+a_9),$$ so that the following is a submatrix of the matrix of syzygies of $J_f$:
$$\left[\begin{array}{cc}
	(-3a_5-3a_9)L_1+6a_2L_2+6a_3L_3+B x&	3L_1 - Ax\\
	6a_4L_1+(-6a_1+3a_5-3a_9)L_2+6a_6L_3+B y&3L_2- Ay\\ 
	6a_7L_1+6a_8L_2+(-6a_1-3a_5+3a_9)L_3+B z&3L_3- Az	
\end{array}\right]_{3\times 2}.$$
Multiplying the second column by $C:=a_1+A=2a_1+a_5+a_9$ and adding it to the first column, we obtain an equivalent matrix
$$\varphi = \left[\begin{array}{cc}
	6(a_1L_1+a_2L_2+a_3L_3)+(B - AC)x&	3L_1 - Ax\\
	6(a_4L_1+a_5L_2+a_6L_3)+(B - AC) y&3L_2- Ay\\
	6(a_7L_1+a_8L_2+a_9L_3)+(B - AC) z&3L_3- Az	
\end{array}\right]_{3\times 2}.$$
Finally, attaching to $\varphi$ a third column corresponding to the Euler derivation, the resulting $3\times 3$ matrix is easily seen to be equivalent to
$$\Phi =\left[\begin{array}{ccc}
	a_1L_1+a_2L_2+a_3L_3&	L_1 &x\\
	a_4L_1+a_5L_2+a_6L_3&L_2&y\\
	a_7L_1+a_8L_2+a_9L_3&L_3&z	
\end{array}\right]_{3\times 3}$$
and satisfies $\mathrm{det}\,\Phi=\frac{1}{2}f$. Now the proposed assertions follow by Lemma \ref{Saito-crit}.
\end{proof}

\medskip

Numerous comments are in order.

\begin{Remark}\rm (a) Recall that, by definition, being reduced is a necessary condition for a polynomial to be free. Now we point out that, in general, it is possible for the cubic $f:={\rm det}\,\Theta$ (with $\Theta$ as defined above) to be non-reduced. For instance, if we start with the matrix $\mathscr{L}(x-y, x+y+z, y+z)$,
then $f=-2(x+z)^3$.

\medskip

\noindent (b) Concerning the condition $f\neq 0$, it means (since ${\rm char}\,k=0$) that the quadrics $Q_1, Q_2, Q_3$ are algebraically independent, hence linearly independent. Clearly, this may not occur; for example, for
$\mathscr{L}(y, x, z)$ we have $f=0$ because $Q_3=-Q_2$.

\medskip

\noindent (c) We remark that any $f$ as in Theorem \ref{fourth-f} is necessarily reducible at least if $k={\mathbb C}$. This follows by the fact that a complex irreducible free divisor in 3 variables must have degree at least 5 (see \cite[Theorem 2.8]{D-S}).
    
\medskip

\noindent (d) A linear free divisor $f$ in our fourth family can have an irreducible quadratic factor, at least over $k={\mathbb R}$ (or eventually a suitable finite field extension of ${\mathbb Q}$). Indeed, starting for example with the matrix
$\mathscr{L}(0, x+z, y+z)$, we obtain $$f \, = \, -2xq \, := \, -2x(x^2+xy-y^2+3xz-yz+z^2).$$ Forcing $q$ to be the product of two linear forms with real coefficients yields a contradiction, hence $q$ is irreducible over ${\mathbb R}$. However, it should be pointed out that $q$ is reducible over ${\mathbb C}$, as the rank of its associated matrix is non-maximal. In fact we believe (but have no proof) that if $k={\mathbb C}$ then a free cubic $f$ as in Theorem \ref{fourth-f} must necessarily be a product of linear forms. This would imply that the complex linear free divisor $z(xz+y^2)$ does not belong to our fourth family, which we have been unable to prove.

\medskip

\noindent (e) Our method does not work for higher degrees in general. Taking for example any of the matrices
$$\left[\begin{array}{ccc}
	x^2&y&z\\
    y^2&z&x
\end{array}\right], \quad \left[\begin{array}{ccc}
	x&y&z\\
    z^2&x^2&y^2
\end{array}\right], \quad \left[\begin{array}{ccc}
	x^2&y^2&z^2\\
    z^2&x^2&y^2
\end{array}\right],$$ 
we are led (following the same recipe) to polynomials that are not free as their Jacobian ideals fail to be perfect. However, an interesting problem remains as to the possibility of producing free divisors by means of a similar technique, but with carefully chosen entries of higher degrees.

\medskip

\noindent (f) Our method does not work for higher dimensions in general. For example, over the ring $k[x, y, z, w]$, consider the matrix
$$\left[\begin{array}{cccc}
	x&y&z&w\\
    x-y&x+w&y-z&x+3y\\
    2y-z&3w&x-w&y+2w
    \end{array}\right].$$ The maximal minors are 4 cubics whose Jacobian matrix has a reduced determinant $g\neq 0$. A computation shows that $J_g$ is not perfect, so that $g$ is not free.

\end{Remark}

We also derive some additional features.

\begin{Proposition}\label{feat-fourth} Let $f$ be as in Theorem \ref{fourth-f}. The following assertions hold:
	\begin{enumerate}
		\item[{\rm (i)}] $J_f$ is of linear type. In particular, ${\rm r}(J_f)=0$;
		\item[{\rm(ii)}] $\mathscr{R}(J_f)$ {\rm (}$\cong {\rm Sym}_RJ_f${\rm )} is a complete intersection;
        \item[(iii)] ${\rm reg}\,{\rm Der}_k(R/(f))=1$.
	\end{enumerate}
\end{Proposition}
\begin{proof} (i) From the proof of the theorem, the $3\times 2$ matrix $\varphi$ is a minimal presentation matrix of $J_f$. It follows easily by the structure of $\varphi$ -- which in particular has only linear forms as entries -- that the so-called $G_3$ condition is satisfied (see the definition in the next section, right before Example \ref{normal-cross}). Moreover, it is clear that $J_f$ has projective dimension 1 (see also Lemma \ref{commut}). Now, applying \cite[Proposition 4.11]{SUV} we obtain that $J_f$ is of linear type.

\medskip

(ii) By the previous item we have $\mathscr{R}(J_f)\cong {\rm Sym}_RJ_f$, and the latter is the quotient of $R[s, t, u]$ (where $s, t, u$ are variables over $R$) by the ideal generated by 2 linear forms $\xi_1, \xi_2$ in $s, t, u$, which are the entries of the matrix product $[s \, \, \,  t \, \, \, u]\cdot \varphi$. Saying that $\hht (\xi_1, \xi_2)=1$ means precisely $$\xi_2 = \lambda \xi_1, \quad \mbox{for some non-zero} \quad \lambda \in k,$$ which is equivalent to the first column of $\varphi$ being $\lambda$ times the second column. Following the proof of the theorem, this would yield ${\rm det}\,\Phi = 0$, a contradiction. Therefore, $\hht (\xi_1, \xi_2)=2$.

\medskip

(iii) Since $f$ is a free cubic, this follows from Lemma \ref{regder}.
\end{proof}

We close the section with a working example which is, on the other hand, somewhat degenerated in the sense that two of the $L_i$'s are equal.

\begin{Example}\rm Taking
$\mathscr{L}(y, x, x)$ yields the line arrangement $$\frac{1}{2}f \, = \, -x^2y+y^3+x^2z-y^2z \, = \, (x+y)(x-y)(z-y),$$ which is then free by Theorem \ref{fourth-f}. In this case, writing down the syzygy matrix $\varphi$ of $J_f$ as in the proof of the theorem and multiplying their columns by suitable non-zero scalars, we get the following simpler presentation matrix for $J_f$: $$\left[\begin{array}{cc}
	y&x\\
    x&y\\
    x&3y-2z
\end{array}\right].$$ It follows by Proposition \ref{feat-fourth}(ii) that the Rees algebra is the complete intersection ring
$$\mathscr{R}(J_f) \, \cong \, R[s, t, u]/(ys+x(t+u),\,xs + yt +(3y-2z)u).$$
\end{Example}

\section{Maximal analytic spread and an application to homaloidness}\label{maxell}

Consider the standard graded polynomial ring $R=k[x_1,\ldots, x_n]=k\oplus R_+$, $n\geq 3$, and let $f\in R$ be a non-zero reduced homogeneous polynomial of degree $d\geq 3$. Recall that the Jacobian ideal $J_f$ can be minimally generated by the derivatives $f_{x_1}, \ldots, f_{x_n}$ since by convention $f$ is not allowed to be a cone. Moreover, $\hht J_f\geq 2$ as $f$ is reduced.

\subsection{When does the Jacobian ideal have maximal analytic spread?}\label{when-max}

Our goal in this part is to answer this question by means of various characterizations. Recall that $$\hht J_f \, \leq \, \ell (J_f) \, \leq \, n,$$ so here we are specifically interested in the property $\ell (J_f)=n$, which holds if for example $J_f$ is of linear type; indeed, in this case we can write $\mathscr{F}(J_f)={\rm Sym}_RJ_f/R_+{\rm Sym}_RJ_f \cong {\rm Sym}_k(J_f/R_+J_f) \cong  R$ as $k$-algebras.

\begin{Example}\rm The Jacobian ideal of the (non-free) cubic $$f \, = \, xyz+w^3 \, \in \,   k[x, y, z, w]$$ can be shown to be of linear type, and so $\ell(J_f)=4$. 
\end{Example}

On the other hand, as we have seen in Proposition \ref{spread4}(2), even for linear free divisors in at least 5 variables the analytic spread of $J_f$ can be arbitrarily smaller than the number $n$ of variables.

Some of the characterizations to be given here are of cohomological nature, and some rely on the asymptotic behavior of depth.
One of the ingredients is a suitable auxiliary module, which we now introduce. As usual, we denote the gradient vector of a polynomial $g\in R$ by $\nabla g=(g_{x_1}, \ldots, g_{x_n})\in R^n$. Given $f$ as above, we set
$$\mathscr{C}_f \, := \, R^n/\left(\sum_{i=1}^nR\,\nabla f_{x_i}\right).$$
A few preparatory concepts are in order before stating our result.

Let $E$ be a finitely generated module over a Noetherian ring $A$ and let $G\stackrel{\Phi}\rar F\rar E\rar0$ be an $A$-free presentation of $E$. Consider the dual map ${\rm Hom}_A(\Phi, A)\colon F\rightarrow G$. The {\it Auslander transpose} (or {\it Auslander dual}) of $E$ is the $A$-module $${\rm Tr}\,E \, = \, {\rm
coker}\,{\rm Hom}_A(\Phi, A),$$ which is unique up to projective summands. We refer to \cite{AB}.

Now, suppose $A=\bigoplus_{i\geq 0}A_i$ is standard graded over a field $A_0$ and let $A_+=\bigoplus_{i\geq 1}A_i$ be the homogeneous maximal ideal of $A$. Assume that the $A$-module $E$ is graded as well. Then, given an integer $j\geq 0$, the {\it $j$-th local cohomology module} of $E$ is the limit 
$$H_{A_+}^j(E) \, = \, \displaystyle \varinjlim{\rm Ext}_A^{j}(A/A_+^s,\,E).$$

Saying that $E\cong E'$ as graded $A$-modules means, as usual, that there is a degree zero isomorphism between $E$ and $E'$.

Finally, given $r\geq 1$, recall that the Noetherian ring $A$ is said to satisfy (Serre's) condition ${\rm S}_r$ if $${\rm depth}\,A_{\mathfrak{p}} \, \geq \, {\rm min}\,\{r,\,\hht \mathfrak{p}\} \quad \mbox{for all} \quad \mathfrak{p}\in {\rm Spec}\,A.$$ This clearly holds (for all $r$) if $A$ is Cohen-Macaulay.

\smallskip

Back to the polynomial setup, our result here is as follows.

\begin{Theorem}\label{max-spread}  Given $f\in R$ as before, the following assertions are equivalent:
	\begin{enumerate}
		\item[{\rm (i)}] $\ell(J_f)=n$;
		\item[{\rm(ii)}] ${\rm dim}\,\mathscr{C}_f=n-1$;
		
		\item[{\rm(iii)}] $\nabla f_{x_1}, \ldots, \nabla f_{x_n}$ are $R$-linearly independent;
		\item[{\rm(iv)}] ${\rm Ext}^1_R({\mathscr C}_f, R)\cong {\mathscr C}_f(d-2)$ as graded $R$-modules; 
	\item[{\rm(v)}] $H_{R_+}^{n-1}({\mathscr C}_f)\cong {\rm Hom}_R({\mathscr C}_f, k)(n-d+2)$  as graded $R$-modules;
	
	\item[{\rm(vi)}] ${\rm depth}\,R/\overline{J_f^{m}}=0$ for some $m \geq 1$, where the bar denotes integral closure;
	
	\item[{\rm(vii)}] ${\rm depth}\,R/\overline{J_f^{m}}=0$ for all $m \gg 0$.
	\end{enumerate} Moreover, if  $\mathscr{R}(J_f)$ satisfies the 
	${\rm S}_2$ condition, then these assertions are also equivalent to the following ones:
	\begin{enumerate}
		\item[{\rm(viii)}] ${\rm depth}\,R/J_f^{m}=0$ for all $m \gg 0$;
		\item[{\rm(ix)}] ${\rm depth}\,R/J_f^{m}=0$ for some $m \geq 1$.
		\end{enumerate}
    \end{Theorem}
\demo Let $H$ be the graded Hessian map of $f$, i.e. the degree zero homomorphism $R^n(-(d-2))\rightarrow R^n$ whose matrix in the canonical bases is the Hessian matrix of $f$. The image of $H$ is the submodule of $R^n$ generated by the homogeneous vectors $\nabla f_{x_1}, \ldots, \nabla f_{x_n}$. Thus, $\mathscr{C}_f$ is the cokernel of $H$, i.e. it has a graded $R$-free presentation 
\begin{equation}\label{H-seq0}
R^n(-(d-2))\stackrel{H}{\longrightarrow} R^n\longrightarrow \mathscr{C}_f\longrightarrow 0.
\end{equation}
Dualizing this sequence, and denoting by $E^*$ (resp. $\phi^*$) the $R$-dual of an $R$-module $E$ (resp. an $R$-module homomorphism $\phi$), we get an exact sequence
\begin{equation}\label{H-seq}
0\longrightarrow \mathscr{C}_f^*\longrightarrow R^n \stackrel{H^*}{\longrightarrow} R^n(d-2)\longrightarrow \mathscr{C}_f(d-2)\longrightarrow 0,
\end{equation}
where we observe that, since the Hessian matrix is symmetric, $H^*=H\otimes {\rm 1}_{R(d-2)}$ so that, indeed, ${\rm coker}\,H^*  =  ({\rm coker}\,H)(d-2)  =  \mathscr{C}_f(d-2)$.

Now, $\ell(J_f)$ is the dimension of the special fiber ring $\mathscr{F}(J_f)=k[f_{x_1}, \ldots, f_{x_n}]$, which by Lemma \ref{rank} can be computed as the rank of the Hessian matrix of $f$. Thus, $$\ell(J_f) \, = \, {\rm rank}\,H \, = \, {\rm rank}\,H^*,$$ and hence $\ell(J_f)=n$ if and only if $H$ is injective (this is of course equivalent to $R^n(-(d-2)) \cong \sum_{i=1}^nR\,\nabla f_{x_i}$ via $H$, which thus proves (i)$\Leftrightarrow$ (iii)), if and only if $H^*$ is injective. The latter property means $\mathscr{C}_f^*=0$. 

Therefore, in order to prove (i)$\Leftrightarrow$ (iv), it suffices to verify that $\mathscr{C}_f^*=0$ if and only if (iv) holds. Suppose $\mathscr{C}_f^*=0$. Then, as we have seen, $H$ is injective. Dualizing (\ref{H-seq0}) (which is now a short exact sequence) and comparing with (\ref{H-seq}), we obtain (iv).
Conversely, assume that (iv) takes place. Thus $$\mathscr{C}_f \, \cong \, {\rm Ext}^1_R({\mathscr C}_f,\,R)(2-d) \, \cong \, {\rm Ext}^1_R({\mathscr C}_f(d-2),\,R).$$ Now recall that the $R$-torsion $\tau_R(\mathscr{C}_f)$ of $\mathscr{C}_f$ coincides with the kernel of the canonical biduality map $\mathscr{C}_f\rightarrow \mathscr{C}_f^{**}$, and so, by \cite[Proposition 2.6(a)]{AB}, we have $$\tau_R(\mathscr{C}_f) \, \cong \, {\rm Ext}^1_R({\rm Tr}\,\mathscr{C}_f, R).$$ But  (\ref{H-seq}) gives ${\rm Tr}\,\mathscr{C}_f=\mathscr{C}_f(d-2)$. Putting these facts together, we obtain
$$\mathscr{C}_f^* \, \cong \, {\rm Ext}^1_R({\mathscr C}_f(d-2),\,R)^* \, \cong \, {\rm Ext}^1_R({\rm Tr}\,\mathscr{C}_f,\,R)^* \, \cong \, \tau_R(\mathscr{C}_f)^* \, = \, 0.$$

Next, let us prove that (iv)$\Leftrightarrow$ (v). The graded canonical module of the standard graded polynomial ring $R$ is $\omega_R=R(-n)$, so $${\rm Ext}^1_R(\mathscr{C}_f,\,\omega_R) \, \cong \, {\rm Ext}^1_R(\mathscr{C}_f,\,R)(-n).$$ Also recall that, in the present setting, the Matlis duality functor is given by ${\rm Hom}_R(-, k)$. Thus, by graded local duality (see \cite[Example 13.4.6]{B-S}), we can write
\begin{equation}\label{Matlis}
H_{R_+}^{n-1}(\mathscr{C}_f) \, \cong \, {\rm Hom}_R({\rm Ext}^1_R(\mathscr{C}_f, R)(-n),\,k) \, \cong \, {\rm Hom}_R({\rm Ext}^1_R(\mathscr{C}_f, R),\,k)(n).
\end{equation}
If (iv) holds, then $H_{R_+}^{n-1}(\mathscr{C}_f) \cong {\rm Hom}_R(\mathscr{C}_f(d-2),\,k)(n) \cong {\rm Hom}_R({\mathscr C}_f, k)(n-d+2)$. Conversely, suppose (v). Using (\ref{Matlis}), we get $${\rm Hom}_R({\mathscr C}_f, k)(n-d+2) \, \cong \, {\rm Hom}_R({\rm Ext}^1_R(\mathscr{C}_f, R),\,k)(n),$$ which is the same as an isomorphism ${\rm Hom}_R({\mathscr C}_f(-n+d-2), k) \cong {\rm Hom}_R({\rm Ext}^1_R(\mathscr{C}_f, R)(-n),\,k)$. Taking Matlis duals and tensoring with $R(n)$, we obtain (iv). 

We proceed to show that (i)$\Leftrightarrow$(ii). First note that, by (\ref{H-seq0}), the 0-th Fitting ideal of $\mathscr{C}_f$ is the principal ideal generated by the determinant $h$ of the Hessian matrix of $f$, so we have
$$\sqrt{0:_R\mathscr{C}_f} \, = \, \sqrt{(h)}$$ and hence
${\rm dim}\,\mathscr{C}_f={\rm dim}\,R/(h)$. It follows that ${\rm dim}\,\mathscr{C}_f=n-1$ if and only if $h\neq 0$, i.e. $H$ is injective, which as seen above is equivalent to (i).

We clearly have $\ell (J_f)=\ell((J_f)_{R_+})$ and $\hht R_+=n$. Thus, by the general characterization given in \cite[Proposition 4.1]{McAdam}, we have that $\ell(J_f)=n$ if and only if $R_+  \in  {\rm Ass}_RR/\overline{J_f^m}$ for all $m\gg 0$, which is tantamount to saying that (vii) holds. This proves the equivalence (i)$\Leftrightarrow$(vii).

Evidently, (vii)$\Rightarrow$(vi), and the converse follows once we recall the chain (see \cite[Proposition 3.4]{McAdam})
$${\rm Ass}_RR/\overline{J_f} \, \subset \, {\rm Ass}_RR/\overline{J_f^2} \, \subset \, {\rm Ass}_RR/\overline{J_f^3} \, \subset \, \ldots$$

Thus, we have proved that the statements (i)$, \ldots,$(vii) are equivalent.

Now we point out that the implication (vii)$\Rightarrow$(viii) holds regardless of $\mathscr{R}(J_f)$ satisfying ${\rm S}_2$. Indeed, condition (vii) means that the irrelevant ideal $R_+$ belongs to the limit value $\overline{\mathscr{A}}^*(J_f)$ of the function $$m \, \mapsto \, {\rm Ass}_RR/\overline{J_f^m},$$ which is known to eventually stabilize (see, e.g., \cite[Proposition 3.4]{McAdam}). There is also the set $\mathscr{A}^*(J_f)$ defined analogously as the stable set of asymptotic prime divisors with respect to the usual filtration given by the powers of $J_f$. By \cite[Proposition 3.17]{McAdam}, we have $$\overline{\mathscr{A}}^*(J_f) \, \subset \, \mathscr{A}^*(J_f)$$ and hence $R_+\in \mathscr{A}^*(J_f)$, which gives (viii). Notice that (viii)$\Rightarrow$(ix) trivially.

It remains to show (ix)$\Rightarrow$(i), under the hypothesis that $\mathscr{R}(J_f)$ satisfies ${\rm S}_2$. In this case, by 
\cite[Remark 2.16]{Ciu}, the extended Rees algebra $$R[J_ft, t^{-1}] \, = \, \bigoplus_{i\in \mathbb{Z}}I^it^i \, \subset \, R[t, t^{-1}]$$ (where, by convention, $I^i=R$ whenever $i\leq 0$) must satisfy ${\rm S}_2$ as well. Note that (ix) means $R_+\in {\rm Ass}_RR/J_f^m$ for some $m\geq 1$. Now we are in a position to apply \cite[Proposition 4.1]{Ciu} in order to conclude that $\ell(J_f)=n$, as needed.
\qed

\begin{Remark}\rm  With the aid of \cite[Proposition 3.26 and Proposition 3.20]{McAdam}, the assertions (i)$, \ldots,$(vii) of Theorem \ref{max-spread} are also seen to be equivalent to each of the following ones:
	\begin{enumerate}
	\item[{\rm(a)}] $R_+\in \overline{\mathscr{A}}^*(IJ_f)$ for any non-zero $R$-ideal $I$, i.e., 
	$${\rm depth}\,R/\overline{I^mJ_f^{m}} \, = \, 0 \quad \mbox{for all} \quad m\gg 0;$$
	\item[{\rm(b)}] For some $j$, the integral closure of $R[f_{x_1}/f_{x_j}, \ldots, f_{x_n}/f_{x_j}]\subset k(x_1, \ldots, x_n)$ contains a prime $\mathfrak{Q}$ of  height $1$ such that $\mathfrak{Q}\cap R=R_+$.
	\end{enumerate} 
\end{Remark}

\smallskip

While, in Theorem \ref{max-spread}, the implication (ix)$\Rightarrow$(i) (and consequently the implication (viii)$\Rightarrow$(i)) holds if the Rees ring $\mathscr{R}(J_f)$ satisfies ${\rm S}_2$, we do not know whether this hypothesis can be dropped. Thus the following question becomes natural (see also Question \ref{quest-homal} in the next subsection).

\begin{Question}\label{when-J-max-spread}\rm Suppose ${\rm depth}\,R/J_f^{m}=0$ for all $m \gg 0$. Is it true that $\ell(J_f)=n$\,?
Does this hold if we only assume that ${\rm depth}\,R/J_f^{m}=0$ for some $m\geq 1$\,?
\end{Question}

Now let $f\in R$ be a linear free divisor. As we will see later, if $n\leq 4$ then $\ell(J_f)=n$. The converse is known to be false, and in Example \ref{normal-cross} below we show in addition that there is a linear free divisor $f$ such that $\ell(J_f)=n$ for any prescribed $n$.

The following well-known notion will be useful (we state it over $R$). A non-zero homogeneous ideal $I$ of $R$, minimally generated by $\nu$ elements, is said to satisfy the $G_s$ condition for a given $s\geq 0$ if $$\hht I_{\nu - j}(\varphi) \, \geq \, j+1 \quad \mbox{for} \quad j=1, \ldots, s-1.$$ Here, $\varphi$ denotes a minimal presentation matrix (or syzygy matrix) of $I$, and note that there is no dependence on the choice of $\varphi$ because each $I_{\nu - j}(\varphi)$ is just a Fitting ideal of $I$.

\begin{Example}\label{normal-cross}\rm Given an arbitrary $n$, consider the normal crossing divisor $$f \, = \, x_1\cdots x_n \, \in \, R \, = \, k[x_1, \ldots, x_n],$$ which is a well-known linear free divisor. The ideal $J_f$ is simply the ideal generated by all the products of distinct $n-1$ indeterminates, and satisfies $${\rm depth}\,R/J_f^m \, = \, {\rm max}\,\{0,\,n-m-1\} \quad \mbox{for all} \quad m\geq 1.$$ In particular, ${\rm depth}\,R/J_f^m=0$ for all $m\geq n-1$. We now claim that $\mathscr{R}(J_f)$ is Cohen-Macaulay (hence it has the ${\rm S}_2$ property). Notice first that a syzygy matrix of $J_f$ is given by
		$$\varphi=\left[\begin{array}{ccccccc}
		x_1& 0&\ldots& 0\\
		0&x_2&\ldots& 0\\
		\vdots&\vdots&\ddots &\vdots\\
		0& 0&\ldots&x_{n-1}\\
		-x_n& -x_n&\ldots& -x_n
		\end{array}\right].$$
Here we have $$\hht I_{n-j}(\varphi) \, = \, j+1 \quad \mbox{for} \quad j=1, \ldots, n-1,$$ so that $J_f$ satisfies the $G_n$ property. Moreover, because $f$ is free, $J_f$ has projective dimension 1 over $R$ (see Lemma \ref{commut}). It follows by \cite[Proposition 4.11]{SUV} that $\mathscr{R}(J_f)$ is Cohen-Macaulay, as claimed. Now we are in a position to apply Theorem \ref{max-spread} to conclude that $\ell(J_f)=n$.

In addition, the theorem gives us that $\mathscr{C}_f$ has projective dimension 1 over $R$ (because the gradient vectors of the natural generators of $J_f$ generate a free module) and dimension $n-1$, hence $\mathscr{C}_f$ is a Cohen-Macaulay module, which yields
\[ H_{R_+}^{i}({\mathscr C}_f)\cong \left\{ \begin{array}{rc}
    {\rm Hom}_R({\mathscr C}_f, k)(2), \, ~~ i=n-1\\\
    0 \hspace{0.96in}, \, ~~   i\neq n-1
    \end{array} \right.
  \]

\end{Example}

\medskip

In Example \ref{normal-cross}, another way to confirm that $\ell(J_f)=n$ is by showing that the (monomial) ideal $J_f$ is of linear type. This fact and lots of other experiments suggest a more restrictive question as well as a conjecture about the interplay between maximal analytic spread and the linear type property; as already seen, the latter implies the former.

\begin{Question}\rm For arbitrary $n\geq 3$ (the number of variables), does there exist a free divisor $f$, with $J_f$ {\it not} of linear type, such that $\ell(J_f)=n$\,?
\end{Question}

The case of interest is $n\geq 4$. Indeed, if $n=3$ then any member of the family of (non-linear) free divisors given in Section \ref{thirdfam} yields an affirmative answer to this question.

\begin{Conjecture} If $f$ is a linear free divisor such that $\ell(J_f)=n$, then $J_f$ is of linear type.

\end{Conjecture}

We have not been able to solve this conjecture for $n\geq 5$. It is true in $n\leq 4$ variables, as we can verify using the classification of linear free divisors given in \cite[6.4, p.\,837]{Grangeretal}.

\smallskip

Next we furnish more examples.

\begin{Example}\rm Consider the so-called Gordan-Noether cubic
$$f=xw^2+ytw+zt^2 \, \in \, R \, = \, k[x, y, z, w, t].$$ In this case, while the symmetric algebra ${\rm Sym}_RJ_f=B/\mathscr{S} =R[t_1, t_2, t_3, t_4, t_5]/\mathscr{S}$ has dimension 6 and depth 5, a calculation shows that $\mathscr{R}(J_f)$ is Cohen-Macaulay (in particular, it has the ${\rm S}_2$ condition). Indeed, if $\varphi$ is a minimal presentation matrix of $J_f$, then the  saturation $$\mathscr{S}:_BI_{4}(\varphi)^{\infty},$$ which by \cite[Lemma 2.11]{Cleto} defines  $\mathscr{R}(J_f)$ in the ring $B$ (note that $I_{4}(\varphi)$ defines the non-principal locus of $J_f$), is perfect of codimension 4. Now, further computations show that ${\rm depth}\,R/J_f=2$ and
$${\rm depth}\,R/J_f^{m} \, = \, 1 \quad \mbox{for all} \quad m\geq 2.$$ Therefore, using Theorem \ref{max-spread}, we conclude that $\ell(J_f)<5$. More precisely, the special fiber ring can be expressed as
$\mathscr{F}(J_f)  = k[t_1, t_2, t_3, t_4, t_5]/(t_2^2-t_1t_3)$, which yields $\ell(J_f)=4$.

\end{Example}

\begin{Example}\rm Consider the quintic $$f=2w^4u+xu^4+ywu^3+zw^2u^2 \, \in \, R \, = \, k[x, y, z, w, u].$$ This is the case $n=5$ of Theorem \ref{free_div_normal_curve}, hence $f$ is a linear free divisor (here it should be mentioned, for completeness, that $f/u$ is not free and its Jacobian ideal is not even linearly presented). By Proposition \ref{spread4}(ii), we have $\ell (J_f) = 4 < 5$, and Theorem \ref{spread4-CM} ensures that $\mathscr{R}(J_f)$ is Cohen-Macaulay. Therefore, by Theorem \ref{max-spread}, we conclude that $${\rm depth}\,R/J_f^{m} \, > \, 0 \quad \mbox{for all} \quad m \geq 1.$$ In fact, for such $f$ it can be verified that ${\rm depth}\,R/J_f=3$,
${\rm depth}\,R/J_f^2=2$, and
${\rm depth}\,R/J_f^m=1$ for all $m\geq 3$. An interesting consequence of the non-vanishing of the asymptotic depth of $J_f$ concerns the higher conormal modules $J_f^m/J_f^{m+1}$. Indeed, in this situation the (also well-defined) conormal asymptotic depth of $J_f$ must be positive as well, since by \cite{Brod} we can write
$$\displaystyle \lim_{m\rightarrow \infty}\,{\rm depth}\,J_f^m/J_f^{m+1} \, \geq \, \displaystyle \lim_{m\rightarrow \infty}\,{\rm depth}\,R/J_f^{m} \, > \, 0.$$
It follows that $R_+\notin {\rm Ass}_RJ_f^m/J_f^{m+1}$ for all $m\gg 0$.

\end{Example}

\begin{Example}\rm Consider the plane sextic $$f=x^6 - 2x^3y^2z + y^4z^2 + y^6 \, \in \, R \, = \, k[x, y, z].$$ This is the case $\alpha=3$ and $\beta =2$ of Theorem \ref{last-family}(i), hence $f$ is a free divisor (which is no longer linear). By Theorem \ref{last-family}(ii), the ring $\mathscr{R}(J_f)$ is Cohen-Macaulay. It can be verified that $${\rm depth}\,R/J_f^m \, = \, 0 \quad \mbox{for all} \quad m\geq 2.$$ Applying Theorem \ref{max-spread} we obtain that $\ell(J_f)=3$. Now from Theorem \ref{last-family}(iii) we know that $J_f$ cannot be of linear type. To see this explicitly, one of the minimal generators of the defining ideal of the Rees algebra in the ring $R[t_1, t_2, t_3]$ is the following polynomial which is not linear in the $t_i$'s:
$$3xyt_1t_2 + 2xzt_1t_3-3yzt_2t_3-3y^2t_3^2-2z^2t_3^2.$$ Furthermore, the theorem yields ${\rm Ext}^1_R({\mathscr C}_f, R)\cong {\mathscr C}_f(4)$ and
\[ H_{R_+}^{j}({\mathscr C}_f)\cong \left\{ \begin{array}{rc}
    {\rm Hom}_R({\mathscr C}_f, k)(-1),  ~~ j=2\\\
    0 \hspace{0.96in}, \, ~~  j\neq 2
    \end{array} \right.
  \]
\end{Example}

\begin{Example}\label{G-S1}\rm Consider the quartic $$f=x^4 - xyz^2 + z^3w \, \in \, R \, = \, k[x, y, z, w],$$ which is not free as $J_f$ is not perfect. Let us also mention that the ideal $J_f$ is not of linear type, since the polynomial $$4xt_2^2 - 4zt_1t_4-yt_4^2 \, \in \, R[t_1, t_2, t_3, t_4]$$ is one of the minimal generators of the defining ideal of $\mathscr{R}(J_f)$. On the other hand, it is not hard to verify that the associated graded ring of $J_f$ -- i.e. the graded algebra $\bigoplus_{s\geq 0}J_f^s/J_f^{s+1}$ -- satisfies the ${\rm S}_1$ property; since in addition $\hht J_f\geq 2$, we get by \cite[Remark 2.16]{Ciu} that $\mathscr{R}(J_f)$ satisfies ${\rm S}_2$ (it can be shown that in fact $\mathscr{R}(J_f)$ is Cohen-Macaulay). Furthermore, ${\rm depth}\,R/J_f^i=1$
for $i=1, 2, 3$, while
$${\rm depth}\,R/J_f^m \, = \, 0 \quad \mbox{for all} \quad m\geq 4.$$ Applying Theorem \ref{max-spread}, we conclude that $\ell(J_f)=4$. We also get ${\rm Ext}^1_R({\mathscr C}_f, R)\cong {\mathscr C}_f(2)$, and
\[ H_{R_+}^{j}({\mathscr C}_f)\cong \left\{ \begin{array}{rc}
    {\rm Hom}_R({\mathscr C}_f, k)(2), \, ~~ j=3\\\
    0 \hspace{0.96in}, \, ~~  j\neq 3
    \end{array} \right.
  \]

\end{Example}

\subsection{Application: Criterion for homaloidness} As above let $f\in R=k[x_1,\ldots, x_n]$, $n\geq 3$, be a non-zero reduced homogeneous polynomial. In this subsection, we assume additionally that the field $k$ is algebraically closed. To the form $f$ we can associate the rational map $$\mathscr{P}_f=(f_{x_1}:\cdots: f_{x_n}):\pp^{n-1}\dasharrow \pp^{n-1},$$ the so-called {\it polar map} defined by $f$. Thus the base locus of $\mathscr{P}_f$ is the singular locus of the projective hypersurface $V(f)\subset \pp^{n-1}$.

\begin{Definition}$($\cite{Dolg}$)$ \rm The polynomial $f$ is {\it homaloidal} if $\mathscr{P}_f$ is birational (hence a Cremona transformation).
\end{Definition}

Over $k={\mathbb C}$, this definition can be translated by saying that $\mathscr{P}_f$ has degree 1 (taking into account an appropriate notion of degree in this context), and according to \cite[Corollary 2]{D-Papa} the property of being homaloidal depends only on $f_{\rm red}$.

The following is a preliminary fact connecting this class of polynomials to the class of free divisors. It can be also seen as a first source of examples of homaloidal divisors (examples in higher dimensions can be found, e.g., in \cite{CRS} and \cite{Maral}). Recall that in general the dimension of the image of the polar map $\mathscr{P}_f$ is given by $\ell(J_f)-1$ (see the proof of Proposition \ref{app-homal}).

\begin{Proposition}\label{n4}{\rm ($k=\mathbb{C}$)} If $n\leq 4$ then every linear free divisor is homaloidal.
\end{Proposition}
\demo Let $f\in R$ be a linear free divisor. Recall we are supposing that $f$ is not a cone (see Subsection~\ref{free divisors}). Thus, by \cite[Proposition 2.4 and Proposition 2.5]{GR}, $f$ has a non-zero Hessian, so that $\ell(J_f)=n$. Hence, the dimension of the image of  $\mathscr{P}_f$ is $n-1$. As the linear rank of the gradient ideal $J_f$ is maximal, it follows by \cite[Theorem 3.2]{AHA} that $\mathscr{P}_f$ is birational.
\qed

\medskip

Notice that this proposition fails if $n\geq 5$. Indeed, if in this case we take $f$ as being a linear free divisor as described in Theorem \ref{free_div_normal_curve}, then by Proposition \ref{spread4}(ii) the analytic spread of $J_f$ is 4, hence the image of $\mathscr{P}_f$ has dimension at most $n-2$ and so this map cannot be birational.

Our application regarding homaloidness is the following ideal-theoretic, also homological, version of the criterion given in \cite[Theorem 3.2]{AHA}. It is not as practical or effective as the original one, but in our view it adds some flavor to the classical -- typically geometric -- theory and, moreover, helps linking to different algebraic tools and invariants.

\begin{Proposition}\label{app-homal}
	Given $f\in R$ as before, let $\varphi_1$ be the submatrix of a minimal syzygy matrix of the ideal $J_f$ consisting of its linear syzygies, and suppose $I_{n-1}(\varphi_1)\neq (0)$. Assume any one of the following situations: 	\begin{enumerate}
	\item[{\rm(i)}] ${\rm projdim}\,\overline{J_f^{m}}=n-1$ for some $m\geq 1$;
	\item[{\rm(ii)}] $\mathscr{R}(J_f)$ satisfies ${\rm S}_2$, and ${\rm projdim}\,J_f^{m}=n-1$ for some $m\geq 1$.
	\end{enumerate} 
	Then $f$ is homaloidal.
	\end{Proposition}
\demo First, in either case, our Theorem \ref{max-spread} (together with the Auslander-Buchsbaum formula) ensures that $\ell(J_f)=n$. On the other hand, 
$${\rm dim}({\rm image}\,\mathscr{P}_f) \, = \, {\rm dim}\,{\rm Proj}\, k[f_{x_1}, \ldots, f_{x_n}] \, = \, {\rm dim}\,{\rm Proj}\,\mathscr{F}(J_f) \, = \, \ell(J_f)-1 \, = \, n-1.$$ Now \cite[Theorem 3.2]{AHA} ensures that $\mathscr{P}_f$ is birational, as needed.
\qed

\begin{Example}\rm Let us first point out that not all homaloidal polynomials satisfy the condition $I_{n-1}(\varphi_1)\neq (0)$. Indeed, consider the cubic
$$f \, = \, xw^2 + yzw + z^3 \, \in \, k[x, y, z, w].$$ Then it can be checked that:
\begin{enumerate}
    \item[{\rm (a)}] $f$ is an irreducible homaloidal polynomial. This is indeed
the first member of the family of  irreducible homaloidal hypersurfaces described in \cite[p.\,1264]{Huh};

\item[{\rm (b)}] The Jacobian ideal $J_f$ is not linearly presented, and moreover has not enough linear syzygies. More precisely, only 2 columns of a minimal presentation matrix $\varphi$ are linear syzygies, and hence obviously $I_{3}(\varphi_1)= (0)$;

\item[{\rm (c)}] $J_f$ is not of linear type;

\item[{\rm (d)}] Quite interestingly, $J_f$ satisfies the conditions present in part (ii) of our Proposition \ref{app-homal}. In particular, it can be even shown that the Rees algebra of $J_f$ is Cohen-Macaulay.
\end{enumerate}
\end{Example}

\medskip

We now remark that if $f$ is a linear free divisor then the condition $I_{n-1}(\varphi_1)\neq (0)$ is automatically satisfied as in this case $\varphi_1=\varphi$ and  $I_{n-1}(\varphi)= J_f$ by the Hilbert-Burch theorem. We thus record the following corollary.

\begin{Corollary}\label{lin-homal} If $f$ is a linear free divisor satisfying either condition {\rm (i)} or {\rm (ii)} of Proposition \ref{app-homal}, then $f$ is homaloidal.
\end{Corollary}

Before giving the first illustration, we raise the following question. We remark that the answer is {\it yes} if the second part of Question \ref{when-J-max-spread} has an affirmative answer as well. Also note that, by Proposition \ref{n4}, the case of interest is $n\geq 5$.

\begin{Question}\label{quest-homal}\rm ($n\geq 5$) Let $f$ be a linear free divisor satisfying ${\rm projdim}\,J_f^{m}=n-1$ for some $m\geq 1$. Must $f$ be homaloidal?

\end{Question}

\begin{Example}\rm The simplest example in arbitrary dimension is the normal crossing divisor $f  =  x_1\cdots x_n \in  R$ studied in Example \ref{normal-cross}. Then $f$ is a linear free divisor and  we have seen in particular that ${\rm depth}\,R/J_f^{n-1}=0$, i.e., ${\rm projdim}\,J_f^{n-1}=n-1$. Since $J_f$ is the ideal generated by all squarefree monomials of degree $n-1$, we get by \cite[Proposition 7.4.5]{Vi} that all powers of $J_f$ are integrally closed; in particular,
$$\overline{J_f^{n-1}} \, = \, J_f^{n-1}.$$
It follows by Corollary \ref{lin-homal} (or
Proposition \ref{app-homal}(i)) that $f$ is homaloidal, thus retrieving the well-known fact that the rational map $\pp^{n-1}\dasharrow \pp^{n-1}$ given by $$(x_1 : \ldots : x_n) \, \mapsto \, (x_2x_3\cdots x_n \, : \, x_1x_3\cdots x_n \, : \, \ldots \, : \, x_1x_2\cdots x_{n-1})$$ is birational -- the so-called Cremona involution on $\pp^{n-1}$.

\end{Example}

Below we illustrate Proposition \ref{app-homal} in the situation where $f$ is not free, and in both  reducible and irreducible cases.

\begin{Example}\label{homal-red}\rm  ($n=6$) Consider the hyperplane-quadric arrangement $$f \, = \, xw(yz + zt + tu) \, \in \,  R \, = \, k[x, y, z, w, t, u].$$ In this case, $f$ is non-free because $J_f$ is not perfect, while on the other hand this ideal (which is linearly presented, so that $\varphi_1=\varphi$) satisfies $I_{5}(\varphi_1)\neq (0)$ and $${\rm projdim}\,J_f^{3} \, = \, 5.$$ Moreover, as in Example \ref{G-S1}, the associated graded ring of $J_f$ has the ${\rm S}_1$ property and hence $\mathscr{R}(J_f)$ satisfies ${\rm S}_2$. By Proposition \ref{app-homal}(ii), $f$ is homaloidal. i.e., the rational map $\pp^{5}\dasharrow \pp^{5}$ given by $$(x : y : z : w : t : u)  \, \mapsto \, (w(yz+zt+tu) \,:\, xzw \,:\, xw(y+t) \,:\, x(yz+zt+tu) \,:\, xwu \,:\, xwt)$$ is Cremona.

\end{Example}

\begin{Example}\rm  ($n=5$) Consider the irreducible cubic $$f \, = \, xt^2 +yzt + z^3+ w^2t \, \in \,  R \, = \, k[x, y, z, w, t].$$ The ideal $J_f$ is perfect but $f$ is non-free as $\hht J_f=3$. It also satisfies $I_{4}(\varphi_1)\neq (0)$ and $${\rm projdim}\,J_f^{3} \, = \, 4.$$ Moreover, the associated graded ring of $J_f$ is Cohen-Macaulay; in particular, $\mathscr{R}(J_f)$ satisfies ${\rm S}_2$. By Proposition \ref{app-homal}(ii), $f$ is homaloidal. Explicitly, the rational map $\pp^{4}\dasharrow \pp^{4}$ given by $$(x : y : z : w : t) \, \mapsto \, (t^2 \,:\, zt \,:\, z^2+\frac{1}{3}yt \,:\, wt \,:\, yz+w^2+2xt)$$ is Cremona.
\end{Example}

Next, we provide a couple of additional observations and questions that, in our view, are interesting and potentially motivating for future research. First, note that if we write the homaloidal quartic $f$ of Example \ref{homal-red} as $$f=xg, \quad g \, = \, w(yz+zt+tu),$$ then a further calculation shows (using again our proposition)  that $g$ is homaloidal as well. This fact, among other examples, led us to suggest the following ``addition-deletion" problem inspired by well-known investigations in free divisor theory (see \cite{Terao}, also \cite{Abe} and \cite{S-T}).

\begin{Question}$($Addition-deletion for homaloidal divisors.$)$ \rm For polynomials $f,g\in R$, with $f$ homaloidal, when is the product $fg$ homaloidal? If $fg$ is  homaloidal, when is $f$ or $g$ homaloidal?

\end{Question}

Now let $f\in R=k[x, y, z, w, t, u, v]$ stand for the 2-catalecticant determinant
$$f = {\rm det }\left[\begin{array}{ccc}
	x&y&z\\
	z&w&t\\
	t&u&v
	\end{array}\right].$$ According to \cite[Proposition 3.25(b)]{Maral}), this cubic is homaloidal. Then, for such $f$, we have detected an intriguing, curious fact: the determinant $h(f)$ of the Hessian matrix of $f$ is a linear free divisor -- in particular, $h(f)$ is already reduced. In the situation where $h(f)$ is not reduced, we naturally consider $h(f)_{\rm red}$, which likewise can be a linear free divisor. For example, let  $$g = {\rm det }\left[\begin{array}{cccc}
	x&w&z&y\\
	y&x&w&z\\
	w&z&y&x\\
	z&y&x&w\\
	\end{array}\right]$$ in the ring $R=k[x, y, z, w]$. Note $g$ is in fact a linear free divisor, and using Corollary \ref{lin-homal} it is not hard to see that $g$ is also homaloidal. Here,
	$$h(g) \, = \, \lambda g^2 \quad \mbox{for some non-zero} \quad \lambda \in k,$$ and therefore $h(g)_{\rm red}$ is free.
	
	As expected, this phenomenon does not take place in general. For instance, if once again we take $f$ as the homaloidal quartic of Example \ref{homal-red}, then a calculation shows $h(f)=3x^2w^2f^2$, so that
	$h(f)_{\rm red}  =  3f$ is not free.

The facts above led us to raise the following question, which reconnects us to the central topic of freeness and closes the paper.

\begin{Question} \rm Let $f$ be a homaloidal polynomial. When is $h(f)_{\rm red}$ a (linear) free divisor?  If $h(f)$ is reduced and not a cone, must it be a (linear) free divisor?
\end{Question}

\bigskip

\noindent{\bf Acknowledgements.} 
The second-named author was supported by the CNPq grants 301029/2019-9 and 406377/2021-9.
The third-named author was supported by the CNPq grants 305860/2019-4 and  425752/2018-6. 

\bibliographystyle{amsalpha}

\end{document}